\documentclass[10pt,reqno]{amsart}
\usepackage{amsmath}
\usepackage{amssymb}
\usepackage{amsthm}
\usepackage{amscd, amsfonts, mathrsfs}

\usepackage{amsaddr}
\usepackage{cases}
\usepackage[dvips]{epsfig}
\usepackage{verbatim} 
\usepackage{epsf}
\usepackage[bookmarksnumbered, pdfpagelabels=true, plainpages=false, colorlinks=true,   linkcolor=black,citecolor=black,urlcolor=black]{hyperref}

\theoremstyle{plain}
\newtheorem{theorem}{Theorem}[section]

\newtheorem{lemma}[theorem]{Lemma}
\newtheorem{proposition}[theorem]{Proposition}

\allowdisplaybreaks

\theoremstyle{definition}
\newtheorem{remark}[theorem]{Remark}
\newtheorem{notation}[theorem]{Notation}

\numberwithin{equation}{section}
\allowdisplaybreaks

\newcommand{\cD}{\mathcal D}

\newcommand{\cH}{\mathcal H}

\newcommand{\ccD}{\mathscr{D}}

\newcommand{\ccP}{\mathscr{P}}
\newcommand{\ccPz}{\mathscr{P}_0}

\newcommand{\al}{\alpha}
\newcommand{\be}{\beta}
\newcommand{\ga}{\gamma}
\newcommand{\Ga}{\Gamma}
\newcommand{\de}{\delta}

\newcommand{\la}{\lambda}

\newcommand{\si}{\sigma}

\newcommand{\Om}{\Omega}

\newcommand{\TT}{\mathbb T}
\newcommand{\RR}{\mathbb R}

\newcommand{\rar}{\rightarrow}

\newcommand{\vol}{\operatorname{vol}}
\newcommand{\id}{\operatorname{id}}
\newcommand{\dive}{\operatorname{div}}
\newcommand{\curl}{\operatorname{curl}}

\newcommand{\norm}[1]{\Vert#1\Vert}
\newcommand{\abs}[1]{\vert#1\vert}
\newcommand{\nnorm}[1]{\left\Vert#1\right\Vert}

\def\cPt{{\tilde{\mathcal P}}}

\def\fractext#1#2{{#1}/{#2}}

\def\indeq{\qquad{}}                     

\title[Free-boundary Euler]{A priori estimates for the free-boundary Euler equations
 with surface tension
 in three
dimensions}
\author[Disconzi]{Marcelo M.~Disconzi}
\address{\vskip -0.25cm Department of Mathematics\\
Vanderbilt University\\Nashville, TN 37240, USA}
\email{marcelo.disconzi@vanderbilt.edu}
\thanks{Marcelo M.~Disconzi is partially supported by NSF grant 1305705, NSF grant 1812826, 
a Sloan Research Fellowship provided by the
Alfred P. Sloan foundation, and a Discovery grant administered by Vanderbilt University.}

\author[Kukavica]{Igor Kukavica}
\address{\vskip -0.25cm  Department of Mathematics\\
University of Southern California \\Los Angeles, CA 91107, USA}
\email{kukavica@usc.edu}
\thanks{Igor Kukavica is partially supported by NSF grant 1615239 and NSF grant 1907992}

\newcommand{\IntDom}{\underset{0 \leq t < T}{\bigcup} \{t\} \times \Om(t)}
\newcommand{\srest}{\mathord{|}}

\expandafter\def\csname r@tocindent4\endcsname{0pt}

\makeatletter
\def\paragraph{\@startsection{paragraph}{4}%
  \z@\z@{-\fontdimen2\font}%
  {\normalfont\it}}
\makeatother

\begin{document}

\maketitle

\begin{abstract}
We derive a priori estimates for the incompressible free-boundary Euler equations
with surface tension
in three spatial dimensions.
Working in Lagrangian coordinates,
we provide a~priori estimates for the local existence when the initial
velocity, which is rotational,
belongs to $H^3$, with some additional regularity on the normal component
of the initial velocity. This lowers the requirement on the regularity of initial
data in the Lagrangian setting.
Our methods are direct and involve
three key elements: estimates for the pressure, the boundary regularity 
provided by the mean curvature, and the Cauchy invariance.
\end{abstract}



\section{Introduction\label{section_intro}}
In this paper we derive a priori estimates for the incompressible free-boundary 
Euler equations with surface tension in three space dimensions. 
While such equations
have been extensively studied,
estimates with optimal regularity are not available. With this eventual goal in mind,
we  provide estimates for the Lagrangian
formulation of the system that close with the velocity in $H^{3}$,
lowering the regularity from previously known results (see the literature review below).

The incompressible free-boundary Euler equations in a domain of $\RR^3$ are given by
\begin{subequations}{\label{free_Euler_system}}
\begin{alignat}{5}
\frac{\partial u}{\partial t}+ \nabla_u u + \nabla p &&\, = \,& \, 0 &&  \hspace{0.25cm}   \text{ in } && \ccD,
\label{free_Euler_eq}
 \\
\dive u  && \, = \, & \, 0 && \hspace{0.25cm}  \text{ in }&& \,  \ccD, 
\label{free_Euler_div}
\\
p  && \, = \,& \, \si \cH  &&  \hspace{0.25cm}  \text{ on } && \, \partial \ccD, 
\label{free_Euler_bry_p} \\
\left. (\partial_t + u_{\alpha} \partial_{x_\alpha})\right|_{\partial \ccD}  && \,  \in \,& T \partial \ccD,  &&\, &&  
\label{free_Euler_bry_u}
\\
u(0, \cdot)  =  u_0,
\hspace{0.2cm}
\Om(0) && \, = \, & \, \Om_0,
\label{free_Euler_ic} 
\end{alignat}
\end{subequations}
where
  \begin{equation*}
    \ccD = \IntDom.   
  \end{equation*}
Above, the quantities $u = u(t,x)$ and $p=p(t,x)$ represent the velocity and 
the pressure of the fluid, $\Om(t) \subset \RR^3$ is the moving (i.e., changing over time) domain,
which may be written as $\Om(t) = \eta(t)(\Om_0)$, where $\eta$ is the flow of $u$,
$\si$ is a non-negative constant
known as the coefficient of surface tension,
and $\cH$ is the mean curvature of the moving (time-dependent) boundary $\partial \Om(t)$.
Note that since $\cH$ is the mean curvature of the embedding of the 
spatial boundary $\partial \Om(t)$ into $\RR^3$, \eqref{free_Euler_bry_p} means
that for each $t$ one has $p(t, \cdot ) = \sigma\cH$.
Also, 
$T\partial \ccD$ is the tangent bundle of $\partial \ccD$,
with $\partial$ denoting the space-time boundary.
The equation \eqref{free_Euler_bry_u} means that 
the boundary $\partial \Om(t)$
moves at a speed equal to the normal component of $u$. 
The quantity $u_0$ is the velocity at time zero
(necessarily divergence-free
by \eqref{free_Euler_div}) and $\Om_0$ is the 
domain at the initial time, assumed with smooth boundary.
The symbol
$\nabla_u$ is the derivative in the direction of $u$, often written as $u \cdot \nabla$.
The unknowns
in \eqref{free_Euler_system} are $u$, $p$, and $\Om(t)$.
Note that $\cH$ and $T\partial \ccD$ are functions of the unknowns 
and thus have to be determined alongside a solution to the problem.

The first existence results for \eqref{free_Euler_system} 
(or,
equivalently, for the equations in Lagrangian coordinates, see
\eqref{Lagrangian_free_Euler_system} below), are those of
Nalimov 
\cite{NalimovCauchyPoisson}
and Yosihara
\cite{YosiharaGravity}, who considered regular irrotational data.
In the case of zero surface tension, Ebin has shown 
in  \cite{Ebin_ill-posed}
that the problem is ill-posed without the Rayleigh-Taylor
stability condition.
The problem of well-posedness under this stability condition and in
the case of zero surface tension was obtained by Wu
\cite{WuWaterWaves2d, 
WuWaterWaves}.
Recently, Wang~et~al have obtained in 
\cite{WZZZ} the local existence under the sharp Sobolev
regularity $H^{2.5+\delta}$ for the zero surface tension case,
extending the previous result of Alazard~et~al
\cite{AlazardSobolevEstimates}, who considered irrotational data.

The case with non-zero surface tension has proven to be more difficult
to 
treat in 
low regularity spaces, at least in the case of rotational fluids. Currently, one does not have estimates
that close in spaces near the threshold $H^{2.5+\delta}$.
In 
\cite{SchweizerFreeEuler},
Schweizer constructed solutions with rotational data in
$H^{4.5}$ with an additional vorticity condition at the surface.
Coutand and Shkoller \cite{CoutandShkollerFreeBoundary} used the Lagrangian formulation and
 constructed solutions with  
$H^{4.5}$ initial data without this restriction.
At the same time, in \cite{ShatahZengGeometry} Shatah and Zeng
obtained a~priori estimates for $H^{3}$ data in Eulerian coordinates
using techniques of infinite dimensional geometry in the spirit of
Ebin and Marsden \cite{EbinMarsden} (see also
\cite{ShatahZengInterface}, where the authors showed how to use their a priori estimates
to obtain a local existence result).
Ignatova and the second author obtained 
in \cite{IgorMihaelaSurfaceTension} 
a~priori estimates with interior regularity in $H^{3.5}$, using
the Lagrangian (direct) approach, while Ebin and the first author
established a local-existence result in $H^{3.5 + \de}$ using a combination
of the Lagrangian approach, infinite-dimensional geometry, and semi-group theory.

For other results on irrotational fluids with surface tension
cf.~\cite{ 
AlazardStabilizationSurfaceTension, 
AlazardCapillaryWaterWaves, 
AlazardWaterWaveSurfaceTension, 
AlazardDispersiveSurfaceTension, 
AmbroseVortexSheets, 
AmbroseMasmoudiWaterWaves, 
Beale_et_al_Growth, 
FeffermanetallSplashSurfaceTension, 
IonescuGlobal3dCapillary, 
GermainMasmoudiShatahGlobalCapillary, 
IfrimTataru2dCapillary, 
IonescuPusateriGlobal2dwaterSurfaceTension, 
YosiharaGravitySurfaceTension}. 
Further related results with non-zero surface tension, including the case of rotational fluids,
vortex sheets, two-phase fluids, and singular limits, are
\cite{
ShkollerVortexSheets, 
CoutandShkollerSplash, 
Disconzilineardynamic,
DisconziEbinFreeBoundary2d,
FeffermanIonescuLie, 
IonescuPusateriGlobal2dwaterModel, 
Ogawa-Tani_FreeBoundarySurfaceTension, 
PusateriTwoPhaseOnePhaseLimitSurfaceTension}. 
Free-boundary problems constitute a very active and fast-growing area of research, and a complete,
or even thorough review of prior works is beyond the scope of this paper. A partial list
of references relevant to the above discussion and the results of this paper is
\cite{
AlazardAboutGlobalExistence,
AlazardCollectionWaterWaves,
AlazardCauchyWaterWaves,
AlazardCauchyTheoryWaterWaves,
AlazardDelortGlobal2dWater,
BieriWu1,
BieriWu2,
FeffermanetallSplash,
FeffermanStructural,
FeffermanetallMuskat,
ShkollerElliptic, 
ChristodoulouLindbladFree, 
CoutandSingularity,
MR2660719,
CraigHamiltonianWaterWaves,
PoryferreEmergingBottom, 
GermainMasmoudiShatahGlobalWaterWaves3D,
IfrimHunterTataru,
IfrimTataruGlobalWater,
IfrimTataruGravityConstant,
Iguchi_et_al_FreeBoundary,
IonescuPusateriWaterWaves2d,
KukavicaTuffaha-Free2dEuler,
KukavicaTuffaha-RegularityFreeEuler,
KukavicaTuffahaVicol-3dFreeEuler,
LannesWaterWaves, 
LannesWaterWavesBook,
LindbladFree1,
Lindblad-LinearizedFreeBoundary,
LindbladNordgren-AprioriFreeBoundary,
Ogawa-Tani_FiniteDepth,
WuAlmostGlobal,
WuGlobal}.

The aim of the present paper is to lower the known initial interior 
regularity for the initial data to $H^{3}$ using the Lagrangian
approach. 
We believe that
the direct nature of the proof 
which we provide
will allow further lowering of the
regularity to the optimal one
$H^{2.5+\delta}$.
So far, this has only been achieved in the irrotational case
by Alazard~et~al in \cite{AlazardWaterWaveSurfaceTension}.

The paper is organized as follows. Section~\ref{section_introduction}
contains the main setup of the problem and the main statement,
Theorem~\ref{main_theorem}.
Section~\ref{section_aux} provides preliminary bounds on the Lagrangian variable and the cofactor matrix.
It also provides some basic geometric results that are used throughout (see 
Lemma~\ref{lemma_geometric}).
The next two sections
provide the pressure estimates and the boundary regularity of the Lagrangian flow.
Section~\ref{section_L_2_estimate} contains the main inequality for system differentiated with 
respect to time three times. 
In Section~\ref{section_time_zero}, we deduce regularity of time derivatives and the pressure at time zero; this is where some of the technical conditions stated in Theorem \ref{main_theorem},
namely, an extra regularity of some quantities on the boundary,
becomes apparent.
Section~\ref{section_H_1_estimate}
contains the main estimate for the twice time differentiated system.
The next section contains the statement on the boundary regularity of the normal component of the velocity. 
Section~\ref{section_div_curl} provides estimates for the divergence and the curl
of the velocity. Finally, in the last section we provide the proof of the main statement.

\section{The system in Lagrangian coordinates and the main result\label{section_introduction}}

We focus on the case $\si > 0$ and consider the model case
\begin{gather}
\Om_0 \equiv \Om = \TT^2 \times (0,1).
\nonumber
\end{gather}
Denoting coordinates on $\Om$ by $(x^1, x^2, x^3)$, set
\begin{gather}
\Ga_1 = \TT^2 \times \{x^3 = 1\}
\nonumber
\end{gather}
and
\begin{gather}
\Ga_0 = \TT^2 \times \{x^3 = 0\},
\nonumber
\end{gather}
so that $\partial \Om = \Ga_0 \cup \Ga_1$. 
Using a change of coordinates, it is easy to adjust the approach to cover the case when the
initial domain is a graph
(cf.~\cite{KukavicaTuffahaVicol-3dFreeEuler});
the estimates remain unchanged modulo lower order terms.
When the domain is not a graph,
it is possible to use a straightening of the boundary and a partition
of unity, as in \cite{CoutandShkollerFreeBoundary}.

We assume that the lower boundary,
which corresponds to the rigid bottom of the fluid domain,
does not move, and thus $\eta(t)(\Ga_0) = \Ga_0$,
where $\eta$ is the flow of the vector field $u$.
We introduce
the Lagrangian velocity and the pressure, respectively,
by $v(t,x) = u(t,\eta(t,x))$ and $q(t,x) = q(t,\eta(t,x))$.
Therefore,
\begin{gather}
\partial_t \eta = v.
\label{eta_dot}
\end{gather}
Denoting by $\nabla$ the derivative
with respect to the spatial variables $x$, introduce the matrix
\begin{gather}
a = ( \nabla \eta )^{-1},
\nonumber
\end{gather}
which is well defined for $\eta$ near the identity. The map $\eta$ is volume-preserving
as a consequence of \eqref{free_Euler_div}, which in turn implies the Piola identity
\begin{gather}
\partial_\be a_{\be\al} = 0.
\label{div_identity}
\end{gather}
(The identity \eqref{div_identity} can be verified by direct computation using
\eqref{a_explicit} below, or cf.~\cite[p.~462]{EvansPDE}.)
Above and throughout, we adopt the following agreement.

\begin{notation}
We denote by $\partial_\al$ spatial derivatives, i.e., 
$\partial_\al = \fractext{\partial}{\partial x_\al}$, for
$\al = 1, 2, 3$. Greek indices ($\al, \be$, etc.) range from $1$ to $3$ and Latin 
indices ($i, j$, etc.), range from $1$ to $2$. Repeated indices are summed over their range.
\end{notation}
In terms of $v$, $q$, and $a$, the system \eqref{free_Euler_system}
becomes
\begin{subequations}{\label{Lagrangian_free_Euler_system}}
\begin{alignat}{5}
\partial_t v_\al + a_{\mu\al} \partial_\mu q &&\, = \,& \, 0 &&  \text{ in } && [0,T) \times \Om,
\label{Lagrangian_free_Euler_eq}
 \\
a_{\al\be}\partial_\al v_\be &&\, = \,& \, 0 &&  \text{ in } && [0,T) \times \Om,
\label{Lagrangian_free_Euler_div}
 \\
 \partial_t a_{\al\be} 
+  a_{\al \ga} \partial_\mu v_\ga a_{\mu \be} &&\, = \,& \, 0 &&  \text{ in } && [0,T) \times \Om,
\label{Lagrangian_free_Euler_a_eq}
\\
a_{\mu \al} N_\mu q + \si |a^T N | \Delta_g \eta_\al &&\, = \,& \, 0 &&  \text{ on } && [0,T) \times \Ga_1,
\label{Lagrangian_bry_q} 
\\
v_\mu N_\mu   &&\, = \,& \, 0 &&  \text{ on } && [0,T) \times \Ga_0,
\label{Lagrangian_bry_v} 
\\
 \eta(0,\cdot) = \id,
\hspace{0.2cm}
 v(0, \cdot)&& \, = \, & \, v_0=(v_{01},v_{02},v_{03}) &&  \text{ in } && \Omega,
\label{Lagrangian_free_Euler_ic} 
\end{alignat}
\end{subequations}
where $\id$ is the identity diffeomorphism on $\Om$, 
$N$ is the unit outer normal to $\partial \Om$, $a^T$ is the transpose of
$a$, $|\cdot|$ is the Euclidean norm, and $\Delta_g$ is the Laplacian on the boundary
with respect to the metric $g_{ij}$ 
induced on
$\partial \Om(t)$ by the embedding $\eta$. Explicitly,
\begin{gather}
g_{ij} = \partial_i \eta_\mu \partial_j \eta_\mu
\label{metric_def}
\end{gather}
and 
\begin{gather}
\Delta_g  (\cdot) = \frac{1}{\sqrt{g}} \partial_i (\sqrt{g} g^{ij} \partial_j (\cdot) ),
\label{Laplacian_def}
\end{gather}
where $g$ is the determinant of the matrix $(g_{ij})$ and $(g^{ij})$ is the inverse matrix. 
In \eqref{Lagrangian_bry_q}, $\Delta_g \eta_\al$ simply means
$\Delta_g$ acting on the scalar function $\eta_\al$, for each $\al =1, 2, 3$.
See Lemma~\ref{lemma_geometric} below for some important identities used to obtain
\eqref{Lagrangian_bry_q}.

Since $\eta(0,\cdot) = \id$, 
the initial Lagrangian and Eulerian velocities agree, i.e., $v_0 = u_0$.
Clearly,
$v_0$ is orthogonal to $\Ga_0$ in view of \eqref{Lagrangian_bry_v}.
Note that
\begin{gather}
a(0,\cdot) = I, 
\nonumber
\end{gather}
where $I$ is the identity matrix, in light of \eqref{Lagrangian_free_Euler_ic}.

\begin{notation}
Sobolev spaces are denoted by 
$H^s(\Om)$, or simply by $H^s$ when no confusion can arise, with the corresponding
norm $\norm{\cdot}_s$; note that $\norm{\cdot}_0$ refers
to the $L^2$ norm. We denote by
$H^s(\partial \Om)$ the Sobolev space of maps defined on $\partial \Om$, with the corresponding
norm $\norm{\cdot}_{s,\partial}$, and similarly the space
$H^s( \Ga_1)$ with the norm  $\norm{\cdot}_{s,\Ga_1}$.
 The $L^p$ norms on $\Om$ and $\Ga_1$ are denoted by
$\norm{\cdot}_{L^p(\Om)}$ 
(or
$\norm{\cdot}_{L^p}$ when no confusion can arise)
and $\norm{\cdot}_{L^p(\Ga_1)}$, respectively.
We use $\srest$ to denote restriction.
\end{notation}

We now state our main result. 

\begin{theorem}
Let $\Om$ be as described above,
and let $\si >0$. 
Let $v_0$ be a smooth divergence-free vector field on $\Om$ and assume that 
$v_0$ is tangent to $\Ga_0$.
Then there exist $T_*>0$ 
and 
a constant $C_0$, depending only on $\si>0$, $\norm{v_0}_3$, 
and $\norm{v_{03}}_{4,\Ga_1}$, 
such that any smooth solution
$(v,q)$ to \eqref{Lagrangian_free_Euler_system} with initial condition
$v_0$ and defined on the time interval $[0,T_*)$, satisfies
\begin{align}
\begin{split}
\norm{v}_3 + \norm{\partial_t v}_{2.5} + \norm{\partial^2_t v}_{1.5}
+ \norm{\partial^3_t v}_0
+ \norm{q}_3 + \norm{\partial_t q}_{2} + \norm{\partial^2_t q}_1 \leq C_0.
\end{split}
   \label{EQ03}
\end{align}
\label{main_theorem}
\end{theorem}

The rest of the paper is dedicated to proving
this statement. We would like to point out that the theorem above
constitutes the first step toward establishing a complete proof of local
existence with the stated regularity. 
In order to obtain such a proof, one should also 
provide the construction of the solution. 
The main additional difficulty in this regard, which will be addressed 
in a future work, is to obtain a formulation of the Cauchy invariance
that is amenable to perturbations.
We would like to add that the adaptation of our results to non-flat
domains is standard. However, we believe that the
necessary straightening of the boundary and  use of a partition of
unity would make the paper more difficult to follow.



The dependence of $T_*$ and $C_0$ on a higher norm of the normal component
of the velocity on the boundary $\Ga_1$
comes from the usual difficulties caused by the moving boundary 
in free-boundary problems.
Note that if $v_0\in H^{4.5}$, as in
\cite{CoutandShkollerFreeBoundary, SchweizerFreeEuler},
then this condition is automatically satisfied.
Also, if the initial data is irrotational
in a neighborhood of the 
interface $\Gamma_1$, 
we have $\Delta v_{03}=0$  there, and
the additional condition on the normal component is not necessary;
cf.~Section~\ref{section_time_zero} below.

For the rest of the paper, 
let $(v,q)$  be
a smooth solution to \eqref{Lagrangian_free_Euler_system}
and assume that 
$\Om$, $\Ga_1$, and $\Ga_0$ are as described above.
%

\section{Auxiliary results}\label{section_aux}

In this section we state some preliminary results that are employed in the proof of 
Theorem~\ref{main_theorem}.

\begin{lemma}
\label{lemma_auxiliary_long}
Assume that $\norm{v}_{3} \leq M$. Then, there exists a constant $C> 0$
such that if $T \in [0,\fractext{1}{C M}]$ and $(v,q)$ is defined on $[0,T]$, 
the following inequalities hold for 
$t \in [0,T]$:\\
\noindent(i) $\norm{\eta}_3 \leq C$.\\
\noindent(ii) $\norm{ a}_2 \leq C$.\\
\noindent(iii) $\norm{ \partial_t a }_{L^p} \leq C \norm{ \nabla v}_{L^p}$, $1 \leq p \leq \infty$.\\
\noindent(iv) $\norm{ \partial_\al \partial_t a }_{L^p} \leq C \norm{\nabla v}_{L^{p_1}}
\norm{\partial_\al a}_{L^{p_2}}  + 
C\norm{\partial_\al \nabla v}_{L^p}$, 
where $\fractext{1}{p} = \fractext{1}{p_1} + \fractext{1}{p_2}$
and $1 \leq p, p_1, p_2 \leq 6$.\\
\noindent(v) $\norm{\partial_t a }_s \leq C \norm{ \nabla v}_s$, $0 \leq s \leq 2$.\\
(vi) $\norm{\partial^2_t a}_s \leq C \norm{\nabla v}_s \norm{\nabla v}_{L^\infty}
+ C \norm{\nabla \partial_t v }_s$ for $0 \leq s \leq 1$, and
$\norm{ \partial^2_t a }_1 \leq C \norm{ \nabla v}^2_{\fractext{5}{4}} + C 
\norm{ \nabla \partial_t v }_1$.\\
\noindent(vii) $\norm{\partial^3_t a}_{L^p} \leq C \norm{\nabla v}_{L^p} \norm{\nabla v}_{L^\infty}^2
+ C \norm{\nabla \partial_t v}_{L^p} \norm{\nabla v}_{L^\infty}
+ C \norm{\nabla \partial^2_t v}_{L^p}$, $1 \leq p < \infty$.\\
\noindent(viii) If $\epsilon\in(0,1]$ 
and $T \leq \fractext{\epsilon}{CM^2}$ then, for 
$t \in [0,T]$, we have
\begin{gather}
\norm{a_{\al \be} - \de_{\al\be} }_2 \leq \epsilon
\nonumber
\end{gather}
and
\begin{gather}
\norm{a_{\al \mu} a_{\be\mu} - \de_{\al\be} }_2 \leq \epsilon.
\nonumber
\end{gather}
Furthermore, if $\epsilon\in(0,1/C]$, 
where $C$ is a sufficiently large constant,
then 
we have the strong ellipticity
\begin{gather}
a_{\al \mu} a_{\be\mu} \xi_\al \xi_\be \geq \frac{1}{C} \abs{\xi}^2.
\nonumber
\end{gather}
\end{lemma}

We emphasize that the constant $C$ above, as well as in the rest of
the paper, does not depend on $M$.

\begin{proof}[Proof of Lemma~\ref{lemma_auxiliary_long}]
The proofs of these inequalities follow directly
from 
\eqref{eta_dot}
and 
\eqref{Lagrangian_free_Euler_a_eq}
as in 
any of the references
\cite{IgorMihaelaSurfaceTension,KukavicaTuffaha-RegularityFreeEuler,KukavicaTuffahaVicol-3dFreeEuler}
and are thus omitted.
\end{proof}

We also recall the following statement from
\cite{IgorMihaelaSurfaceTension} on the $H^{1}$ 
regularity of the Neumann problem.

\begin{lemma}[\cite{IgorMihaelaSurfaceTension}]
\label{lemma_elliptic_estimate}
Let $f$ be an $H^1$ solution to
\begin{alignat}{5}
\partial_\al( b_{\al\be}\partial_\be f)  &&\, = \,& \, \partial_\al\pi_\al &&  \text{ in } &&  \Om,
\nonumber
\\
b_{\al\be} \partial_\al f N_\be &&\, = \,& \, g  &&  \text{ on } && \partial \Om,
\nonumber
\end{alignat}
where $b \in H^2(\Om)$, $b_{\al\be} = b_{\be\al}$, 
$\pi, \partial_\al \pi_\al \in L^2(\Om)$, 
and $g \in H^{-\frac{1}{2}}(\partial \Om)$, with the compatibility condition
\begin{gather}
\int_{\partial \Om} (\pi_\mu N_\mu - g ) = 0.
\nonumber
\end{gather}
Suppose that $\norm{b}_{L^\infty} \leq M$ and 
$b_{\al\be}\xi_\al \xi_\be \geq \abs{\xi}^2/M$. If
\begin{gather}
\norm{ b - I }_{L^{\infty}} \leq \epsilon_0,
\nonumber
\end{gather}
where $I$ is the identity matrix and $\epsilon_0$ is a sufficiently small constant depending
on $M$, then
\begin{gather}
\norm{f }_1 \leq 
C \norm{ \pi }_{L^2{\Om}} + C \norm{ g - \pi_\mu N_\mu }_{-\frac{1}{2},\partial}
+ \norm{f}_{L^2(\Ga_1)}.
\nonumber
\end{gather}
\end{lemma}


\begin{notation}
In the rest of the paper, the symbol $C$ denotes a positive sufficiently large
constant.
It can vary 
from expression to expression, but it is always  independent of 
$\norm{v}_3$, $\norm{\partial_t v}_{2.5}$, $\norm{\partial^2_t v}_{1.5}$,
$\norm{\partial^3_t v}_0$, $\norm{q}_3$,
$\norm{\partial_t q}_{2}$, and $\norm{\partial^2_t q}_1$.
The a~priori estimates require $T$ to be sufficiently small so
that it satisfies
$T M\le 1/C$, where $M$ is an upper bound on the norm of the solution
(cf.~Lemma~\ref{lemma_auxiliary_long} above).
In several estimates it suffices to keep track only on the number of derivatives so we 
write $\partial^\ell$ to denote any (spatial) derivative of order $\ell$ and $\overline{\partial}{}^\ell$
to denote any (spatial) derivative of order $\ell$ on the boundary, i.e., with respect to $x_i$.
We use upper-case Latin indices to denote $x_i$ or $t$, so $\overline{\partial}_A$
means $\partial_t$ or $\partial_i$.
\end{notation}

\begin{remark}
\label{remark_rational}
Throughout the paper, we have to estimate several rational functions
of tangential derivatives of $\eta$ that come from various combinations of $g^{ij}$ and
$\sqrt{g}$, where by tangential derivatives we mean $\partial_1$ and
$\partial_2$. Although there will be several different such expressions,
they can all be organized in the same fashion. Therefore, it is beneficial to introduce some general
notation and explain how these terms will be handled in a unified manner. We also take this opportunity
to introduce some symbolic notation that will help us keep track of the number of derivatives.

Consider a rational function $Q$ of the (tangential) derivatives of $\eta$:
\begin{gather}
Q = Q(\partial_1 \eta_1, \partial_2 \eta_1,
\partial_1 \eta_2, \partial_2 \eta_2,
\partial_1 \eta_3, \partial_2 \eta_3).
\nonumber
\end{gather}
We assume that $Q$ is defined in an open domain $\cD$ of $\RR^6$.
Assume that $0 \notin \overline{\cD}$ and that $(1,0,0,1, 0, 0) \in \cD$. 
This last assumption reflects the fact that at $t=0$ we have $\eta = \id$, so that
$\partial_1 \eta_1 = \partial_2 \eta_2 = 1$ at $t=0$, while 
the other derivatives are zero.
In what follows it suffices to keep track of the generic form of some expressions so 
we write $Q$ symbolically as 
 \begin{gather}
 Q = Q(\overline{\partial} \eta).
 \nonumber
 \end{gather}
Then
\begin{gather}
\overline{\partial}_A Q (\overline{\partial} \eta)= \widetilde{Q}^i_{\al} (\overline{\partial} \eta)
\overline{\partial}_A \partial_i \eta_\al,
\nonumber
\end{gather}
where the terms $\widetilde{Q}^i_{\al} (\overline{\partial} \eta)$ are also 
rational function of derivatives of $\eta$ with 
respect to $x_i$. Note that
$\widetilde{Q}^i_{\al} (\overline{\partial} \eta)$
are simply the partial derivatives of $Q$ evaluated at $\overline{\partial} \eta$.
We write the last equality symbolically as
\begin{gather}
\overline{\partial}_A Q(\overline{\partial} \eta) =
\widetilde{Q}(\overline{\partial} \eta) \overline{\partial}_A \overline{\partial} \eta.
\label{partial_Q_symb}
\end{gather}
Regarding $\widetilde{Q}$, we assume that for all $\overline{\partial} \eta$ taking values
in some some small neighborhood $\cD^\prime$ 
of $(1,0,0,1,0,0)$, and for any $1 < s \leq 1.5$,
we have
\begin{align}
\norm{\widetilde{Q}(\overline{\partial} \eta)}_{s,\Ga_1} \leq C
\label{Q_sym_estimate}
\end{align}
for some constant $C$ depending only on $\cD^\prime$.

For $s > 1$, we have 
\begin{gather}
\norm{\overline{\partial}_A Q(\overline{\partial} \eta)}_{s,\Ga_1} \leq C_1
\norm{\widetilde{Q}(\overline{\partial} \eta)}_{s,\Ga_1}
\norm{ \overline{\partial}_A \overline{\partial} \eta}_{s,\Ga_1}
\nonumber
\end{gather}
where $C_1$ depends only on $s$ and on the domain $\Ga_1$.

The term $\norm{ \widetilde{Q}(\overline{\partial} \eta)}_{s,\Ga_1}$
can be estimated using the Sobolev norm of the map $\widetilde{Q}$, i.e.,
 $\norm{\widetilde{Q}}_{H^s(\cD^\prime)}$, and the Sobolev norm of  $\overline{\partial} \eta$, i.e., 
$\norm{\overline{\partial} \eta}_{s,\Ga_1}$.
Under the conditions of Lemma~\ref{lemma_auxiliary_long}, we have 
\begin{gather}
\norm{ \overline{\partial} \eta - \overline{\partial} \eta (0)}_{L^\infty(\Ga_1)}
\leq \int_0^t \norm{\partial_t \overline{\partial} \eta }_{L^\infty(\Ga_1)}
\leq C_2 t \norm{v}_3 \leq C_2 Mt,
\nonumber
\end{gather}
where $C_2$ depends only on the domain $\Ga_1$ and we used that
$H^{1.5}(\Ga_1)$ embeds into $C^0(\Ga_1)$. 
Therefore, if $C M T \leq 1$,
we can guarantee that 
\begin{gather}
\overline{\partial}\eta(\Ga_1) \subset \cD^\prime,
\nonumber
\end{gather}
and thus, shrinking $\cD$ if necessary, we can 
apply the estimate \eqref{Q_sym_estimate}.
Since Lemma~\ref{lemma_auxiliary_long} also provides a bound for 
$\norm{\overline{\partial} \eta}_{s,\Ga_1}$, with $s \leq 1.5$, we conclude that 
\begin{gather}
\norm{\overline{\partial}_A Q(\overline{\partial} \eta)}_{s,\Ga_1} \leq C
\norm{ \overline{\partial}_A \overline{\partial} \eta}_{s,\Ga_1},
\, \text{ for } \, 1 < s \leq 1.5,
\label{rational_estimate}
\end{gather}
where $C$ depends only on $M$, $s$, and $\Ga_1$, and provided that $t$ is small
enough. The above also shows that 
\begin{gather}
\norm{ Q(\overline{\partial} \eta)}_{s,\Ga_1} \leq C
\norm{  \overline{\partial} \eta}_{s,\Ga_1},
\, \text{ for } \, 1 < s \leq 1.5
   .
\label{rational_estimate_2}
\end{gather}
Above, the restriction $s  > 1$ comes from the Sobolev multiplicative theorem, while
$s \leq 1.5$ allows us to invoke the estimates of Lemma~\ref{lemma_auxiliary_long}.
\end{remark}

We also need some geometric identities that may be known to specialists,
but we state them below and provide some of the corresponding proofs
for the reader's convenience.

\begin{lemma}
\label{lemma_geometric}
Denote by $n$ the unit outer normal to $\eta(\Ga_1)$. Then
\begin{gather}
n\circ \eta = \frac{a^T N}{|a^T N|}.
\label{normal_identity}
\end{gather}
Denoting by $\tau$ the tangent bundle of $\overline{\eta(\Om)}$ and by
$\nu$ the normal bundle of $\eta(\Ga_1)$, the canonical projection
$\Pi\colon \tau \srest \eta(\Ga_1) \rar \nu$ is given by
\begin{gather}
\Pi_{\al\be} = \de_{\al\be} - g^{kl} \partial_k \eta_\al \partial_l \eta_\be.
\label{projection_identity}
\end{gather}
Furthermore, the following identities hold:
\begin{gather}
\Pi_{\al\la} \Pi_{\la\be} = \Pi_{\al\be},
\label{projections_contraction}
\end{gather}
\begin{gather}
|a^T N | = \sqrt{g},
\label{aT_identity}
\end{gather}
\begin{gather}
\sqrt{g} \Delta_g \eta_\al = \sqrt{g} g^{ij}  \partial^2_{ij} \eta_\al 
- \sqrt{g} g^{ij} g^{kl} \partial_k\eta_\al \partial_l \eta_\mu \partial^2_{ij} \eta_\mu,
\label{Laplacian_eta_identity}
\end{gather}
\begin{align}
\begin{split}
\overline{\partial}_A(\sqrt{g} \Delta_g \eta_\al ) 
 = & \, \partial_i \Big( \sqrt{g} g^{ij} (\de_{\al\la} -g^{kl} \partial_k \eta_\al \partial_l \eta_\la)
\overline{\partial}_A \partial_j \eta_\la  
\\
&
\qquad\qquad+ \sqrt{g}(g^{ij} g^{kl} - g^{lj}g^{ik} ) \partial_j \eta_\al \partial_k\eta_\la 
\overline{\partial}_A \partial_l \eta_\la \Big),
\end{split}
\label{partial_Laplacian_eta_identity}
\end{align}
\begin{align}
\begin{split}
-\Delta_g (\eta_\al \srest \Ga_1) = \cH \circ \eta \, n_\al \circ \eta.
\end{split}
\label{formula_mean_curvature_embedding}
\end{align}
\end{lemma}

\begin{proof}[Proof of Lemma~\ref{lemma_geometric}]
Letting $r = \eta \srest \Ga_1$, we know that $n \circ \eta$ is given by
\begin{gather}
n \circ \eta = \frac{\partial_1 r \times \partial_2 r}{|\partial_1 r \times \partial_2 r|}
\label{normal_standard_formula}
\end{gather}
(cf.~e.g.~\cite{HanIsometricEmbedding}).
By $\det(\nabla \eta) =1$, we have
\begin{align}
\begin{split}
a = 
\begin{bmatrix}
\partial_2 \eta_2 \partial_3 \eta_3 - \partial_3 \eta_2 \partial_2 \eta _3
& \partial_3 \eta_1 \partial_2 \eta_3 - \partial_2 \eta_1 \partial_3 \eta_3
&
\partial_2 \eta_1 \partial_3 \eta_2 - \partial_3 \eta_1 \partial_2 \eta_2
\\
\partial_3 \eta_2 \partial_1 \eta_3 - \partial_1 \eta_2 \partial_3 \eta_3
&
\partial_1 \eta_1 \partial_3 \eta_3 - \partial_3 \eta_1 \partial_1 \eta _3 
&
\partial_3 \eta_1 \partial_1 \eta_2 - \partial_1 \eta_1 \partial_3 \eta_2
\\
\partial_1 \eta_2 \partial_2 \eta_3 - \partial_2 \eta_2 \partial_1 \eta_3
&
\partial_2 \eta_1 \partial_1 \eta_3 - \partial_1 \eta_1 \partial_2 \eta _3 
&
\partial_1 \eta_1 \partial_2 \eta_2 - \partial_2 \eta_1 \partial_1 \eta_2
\end{bmatrix}.
\end{split}
\label{a_explicit}
\end{align}
Using \eqref{a_explicit} to compute $a^T N$ and comparing with 
 $\partial_1 r \times \partial_2 r$, one verifies that  
$
a^T N = \partial_1 r \times \partial_2 r
$,
and then \eqref{aT_identity} follows from \eqref{normal_standard_formula}.

To prove \eqref{projection_identity}, we use \eqref{normal_identity} 
to write
\begin{align}
\begin{split}
(\de_{\al\la} - g^{kl} \partial_k \eta_\al \partial_l \eta_\la)
n_\la \circ \eta 
= \frac{ a_{\mu \al} N_\mu }{ |a^T N |} -
\frac{g^{kl} \partial_k \eta_\al \partial_l \eta_\la a_{\mu \la} N_\mu }{|a^T N|}.
\end{split}
\nonumber
\end{align}
Contracting $g^{kl} \partial_l \eta_\la a_{\mu \la} N_\mu$
with $g_{mk}$ gives
\begin{align}
\begin{split}
g_{mk} g^{kl}\partial_l \eta_\la a_{\mu \la} N_\mu  
= &\, \partial_m \eta_\la a_{3 \la}
\\
 = & \,
\partial_m \eta_1
(\partial_1 \eta_2 \partial_2 \eta_3 - \partial_2 \eta_2 \partial_1 \eta_3)
+ 
\partial_m \eta_2
(\partial_2 \eta_1 \partial_1 \eta_3 - \partial_1 \eta_1 \partial_2 \eta_3)
\\
&
+
\partial_m \eta_3
(\partial_1 \eta_1 \partial_2 \eta_2 - \partial_2 \eta_1 \partial_1 \eta_2)
=0.
\end{split}
   \label{EQ01}
\end{align}
Above, the first equality follows because $N = (0,0,1)$ (and $g_{mk} g^{kl} = \de_m^l$),
the second equality uses~\eqref{a_explicit}, and the third equality follows upon setting
$m=1$ and then $m=2$ and observing that in each case all the terms cancel out.
Thus, contracting \eqref{EQ01} with $g^{m n}$, we obtain
$
g^{nl}\partial_l \eta_\la a_{\mu \la} N_\mu  =  0
$,
and hence
\begin{align}
\begin{split}
(\de_{\al\la} - g^{kl} \partial_k \eta_\al \partial_l \eta_\la)
n_\la \circ \eta 
= \frac{ a_{\mu \al} N_\mu }{ |a^T N |}.
\end{split}
\nonumber
\end{align}
To conclude the proof of \eqref{projection_identity}, we need to verify that $\Pi(X) = 0$
if $X$ is tangent to $\eta(\Ga_1)$. Since the tangent space to  $\eta(\Ga_1)$ is spanned
by $\partial_j \eta$, for $j=1,2$, it suffices to verify the identity
for these vectors. We have
\begin{align}
\Pi_{\al \mu} \partial_j \eta_\mu & = (\delta_{\al\mu} - g^{kl}\partial_k \eta_\al \partial_l \eta_\mu)
\partial_j \eta_\mu 
\nonumber
=
\partial_j \eta_\al - 
g^{kl}\partial_k \eta_\al g_{lj} 
 = 0,
\end{align}
where we used  $g_{lj} = \partial_l \eta_\mu \partial_j \eta_\mu$ and 
$g^{kl} g_{lj} = \delta^k_j$.
Thus, \eqref{projection_identity} is proven.

The identity~\eqref{projections_contraction} follows 
from the fact that $\Pi$ is a projection operator
or, alternatively, by direct computation using \eqref{projection_identity}.
The identity~\eqref{aT_identity} follows from \eqref{normal_identity}, \eqref{normal_standard_formula},
and the standard formula 
\begin{gather}
\frac{\partial_1 r \times \partial_2 r}{|\partial_1 r \times \partial_2 r|} =
\frac{1}{\sqrt{g}} \partial_1 r \times \partial_2 r
 \nonumber
 \end{gather}
(see e.g.~\cite{HanIsometricEmbedding}).
In order to prove \eqref{Laplacian_eta_identity}, recall that 
(see e.g.~\cite{HanIsometricEmbedding})
\begin{gather}
\Delta_g \eta_\al = g^{ij} \partial^2_{ij} \eta_\al - g^{ij} \Ga^k_{ij} \partial_k \eta_\al,
\label{Laplacian_identity_standard}
\end{gather}
where $\Ga^k_{ij}$ are the Christoffel symbols.
Recalling \eqref{metric_def},
a direct computation using the definition of the Christoffel symbols
gives 
\begin{gather}
\Ga_{ij}^k = g^{kl}\partial_l \eta_\mu \partial^2_{ij} \eta_\mu,
\label{Christoffel_identity}
\end{gather}
and \eqref{Laplacian_eta_identity} follows from \eqref{Laplacian_identity_standard}
and \eqref{Christoffel_identity}. 

Now, we  move to establish \eqref{partial_Laplacian_eta_identity}.
Using \eqref{Laplacian_def},
\begin{align}
\begin{split}
\overline{\partial}_A (\sqrt{g} \Delta_g \eta_\al )
&= 
\overline{\partial}_A \partial_i (\sqrt{g} g^{ij} \partial_j \eta_\al ) 
 = \partial_i( \sqrt{g} g^{ij} \partial_j \overline{\partial}_A \eta_\al 
+ \overline{\partial}_A(\sqrt{g} g^{ij} ) \partial_j \eta_\al ).
\end{split}
\label{derivation_partial_Laplacian_1}
\end{align}
Recalling the 
standard identity (see e.g.~\cite[p.~51]{GieriBook}),
\begin{gather}
\overline{\partial}_A g = g g^{kl} \overline{\partial}_A g_{kl},
\label{identity_det_g}
\end{gather}
which is Jacobi's identity connecting the 
derivative of a determinant with a trace,
we find
\begin{gather}
 \overline{\partial}_A(\sqrt{g} g^{ij} )  = \sqrt{g}
  \left(\frac{1}{2} g^{ij}   g^{kl} - g^{lj} g^{ik} \right)  \overline{\partial}_A g_{kl} , 
    \nonumber
  \end{gather}
where we also used 
\begin{gather}
 \overline{\partial}_A g^{ij} = - g^{lj} g^{ik}  \overline{\partial}_A g_{kl}
 \label{identity_g_inverse}
 \end{gather}
that follows from differentiating $g_{kl}g^{lj} = \de_k^j$ and then contracting with
$g^{ik}$. Computing $\overline{\partial}_A g_{kl}$ directly from \eqref{metric_def}
then leads to
\begin{gather}
 \overline{\partial}_A(\sqrt{g} g^{ij} )  = \sqrt{g}
  \left(\frac{1}{2} g^{ij}   g^{kl} - g^{lj} g^{ik} \right)  ( \overline{\partial}_A \partial_k \eta_\la 
  \partial_l \eta_\la + \partial_k \eta_\la \overline{\partial}_A \partial_l \eta_\la  )
   .
\label{derivation_partial_Laplacian_2}
  \end{gather}
Using \eqref{derivation_partial_Laplacian_2} in  \eqref{derivation_partial_Laplacian_1}
then yields \eqref{partial_Laplacian_eta_identity} after some 
computation.

Finally, the identity \eqref{formula_mean_curvature_embedding} is a standard formula
for the mean curvature of an embedding into $\RR^3$ 
(see e.g.~\cite[Section~2.1]{HanIsometricEmbedding} or \cite[Section~I.2]{GieriBook}).
\end{proof}

\section{Pressure estimates}\label{section_pressure}


Through the rest of the paper, we suppose that the hypotheses of
Lemma~\ref{lemma_auxiliary_long} hold. 
 We assume further that $T$ is as in part
(viii) of that lemma, and that $(v,q)$ are defined on $[0,T)$. 


\begin{proposition}
\label{proposition_pressure_estimates}
We have the estimates
\begin{align}
\begin{split}
\norm{ q }_3  \leq &  \, C \norm{ \nabla v}_2 \norm{ v }_2 + C \norm{ \partial_t v }_{1.5,\Ga_1} + C,
\\
\norm{ \partial_t q}_{2} \leq & \, C \norm{ \nabla v}_{1.5+\de} (\norm{q}_{2.5}
+ \norm{\partial_t v}_{1.5}) + C(\norm{\nabla v}_{1.5} \norm{\nabla v}_{L^\infty} +
\norm{\nabla \partial_t v}_{1.5} ) \norm{ v }_{1.5+\de}
\\
&
+ C \norm{\partial^2_t v}_{1,\Ga_1} + C \norm{v}_{2.5} + C \norm{\partial_t v }_{2.5},
\\
\norm{\partial^2_t q}_1  \leq & \, C (\norm{v}_{1.5} \norm{\nabla v}_{L^\infty} + 
\norm{\partial_t v}_{1.5})( \norm{q}_2 + \norm{\partial_t v}_1 )
+
C \norm{\nabla v}_{L^\infty} ( \norm{\partial_t q}_1 + \norm{\partial^2_t v}_0 )
\\
& 
+ C (\norm{v}_{1.5} \norm{ \nabla v}^2_{L^\infty} + \norm{\partial_t v}_{1.5} \norm{\nabla v}_{L^\infty} + \norm{\partial^2_t v}_{1.5} ) \norm{ v}_1 + C \norm{ \partial^3_t v}_0
\\
& 
+ C (\norm{\partial_t v}_{2.5} + \norm{v}_3 \norm{\partial_t q}_1 )
+ C(1+\norm{v}^2_3 + \norm{\partial_t v}_2)(\norm{\partial_t v}_2 + \norm{v}_3^2 ),
\end{split}
\nonumber
\end{align}
for $t \in (0,T)$ and with $\de > 0$ a small number.
\end{proposition}

The pressure estimates are performed similarly to \cite{IgorMihaelaSurfaceTension}. The
differences
are the spaces in which we estimate $q$ and $\partial_{t}q$.
The inequality for $\partial_{t}^2q$ is the same as in 
\cite{IgorMihaelaSurfaceTension}, except for the bound on 
$\Vert \partial_{t}^2q\Vert_{0,\Gamma_1}$, which 
we state in \eqref{EQ05} below.
The adjustment has to be made to account for the evolution of the Riemannian
metric on the boundary.

\begin{proof}[Proof of Proposition~\ref{proposition_pressure_estimates}]
Contracting $a_{\mu \al}\partial_\mu$ with \eqref{Lagrangian_free_Euler_eq} and using \eqref{Lagrangian_free_Euler_div} we find
\begin{gather}
\Delta q = \partial_\mu( (\de_{\mu\la} - a_{\mu\nu} a_{\la \nu} ) \partial_\la q ) + \partial_\mu ( \partial_t a_{\mu\la} v_\la ),
\label{press_estimate_eq}
\end{gather}
where we also used \eqref{div_identity}.
Contracting \eqref{Lagrangian_free_Euler_eq} with $a_{\mu \al} N_\mu$ and restricting to the boundary yields
\begin{gather}
\frac{\partial q}{\partial N} = ( \de_{\mu \la} - a_{\mu\nu} a_{\la \nu} )N_\mu \partial_\la q - a_{\mu \nu} N_\mu \partial_t v_\nu
\, \text{ on } \, \partial \Om.
\label{press_estimate_bry}
\end{gather}
Denoting the right-hand sides of \eqref{press_estimate_eq} and \eqref{press_estimate_bry} by $f$ and $g$, respectively,
we have a standard elliptic
estimate 
\begin{gather}
\norm{q}_3 \leq C \norm{f}_1 + C \norm{g}_{1.5,\partial} + \norm{q}_0
\nonumber
\end{gather}
(see e.g.~\cite[Theorem~3.6]{ShkollerElliptic}).
With the help of Lemma~\ref{lemma_auxiliary_long}, we find
\begin{gather}
\norm{f}_1 \leq \epsilon \norm{q}_3 + C \norm{\nabla v}_2 \norm{v}_2
\nonumber
\end{gather}
and 
\begin{gather}
\norm{g}_{1.5,\partial} \leq \epsilon \norm{q}_3 + C \norm{\partial_t v}_{1.5,\Ga_1}
   .
\nonumber
\end{gather}
To estimate $\norm{q}_0$, we use 
\begin{gather}
\left| \int_\Om q \,  \right| \leq C \norm{ \nabla(q - \overline{q} ) }_0 + C \norm{q}_{0,\Ga_1},
\nonumber
\end{gather}
which follows from the identity
\begin{gather}
\int_\Om q
   = \int_\Om q \Delta h
   = - \int_\Om \nabla q \cdot \nabla h + \int_{\partial \Om} \frac{\partial h}{\partial N} q
\nonumber
\end{gather}
and the choice
$h=x_3^2-1$.
Using \eqref{Lagrangian_bry_q}
and Lemma~\ref{lemma_auxiliary_long}, we obtain$\norm{q}_{0,\Ga_1} \leq C$. Combining the above inequalities
gives the estimate for $q$ after a simple application of the $\epsilon$-Cauchy inequality.

The estimate for $\partial_t q$ is obtained similarly after time-differentiating \eqref{Lagrangian_free_Euler_eq}
and proceeding as above.

To obtain the estimate for $\partial_t^2 q$, we differentiate \eqref{Lagrangian_free_Euler_eq} in time
twice and
apply Lemma~\ref{lemma_elliptic_estimate}. This is 
done similarly to \cite{IgorMihaelaSurfaceTension},
with an exception of the estimate on 
$\norm{\partial_{t}^2q}_{0,\Ga_1}$ which reads
\begin{align}
\begin{split}
\norm{\partial^2_t q}_{0,\Ga_1}  \leq & \, \epsilon \norm{\partial^2_t q}_1 + 
C (\norm{\partial_t v}_{2.5} + \norm{v}_3 \norm{\partial_t q}_1 )
+ C(1+\norm{v}^2_3 + \norm{\partial_t v}_2)(\norm{\partial_t v}_2 + \norm{v}_3^2 )
\end{split}
   \label{EQ05}
\end{align}
leading to the last two terms in the statement.
\end{proof}

\section{Regularity estimate for the flow\label{section_regularity}}
One of the key features of the free-boundary Euler equations with surface tension
is a gain of regularity 
for the free-boundary \cite{CoutandShkollerFreeBoundary, DisconziEbinFreeBoundary3d}.
As discussed in \cite{ShatahZengGeometry}, this gain is geometric in nature and does not
correspond to the regularity of the flow in the interior (see the counter-example
in \cite{ShatahZengGeometry}, which shows that in the interior of the domain, 
the flow cannot in general be more regular than
the velocity, even if the boundary gains regularity). This gain of regularity 
is formulated in 
the next statement.

\begin{proposition}
\label{proposition_regularity}
We have the estimate
\begin{gather}
\norm{\eta}_{3.5,\Ga_1} \leq P(\norm{q}_{1.5,\Ga_1}),
\nonumber
\end{gather}
where $P$ is a polynomial.
\end{proposition}

\begin{notation}
From here on, we use $P(\cdot)$, with indices attached when appropriate, to denote a 
generic polynomial expression of its arguments. The polynomial is assumed
nonnegative and its value may change from inequality to inequality.
Also, we write
\begin{gather}
\ccP = P( \norm{v}_3,  \norm{\partial_t v}_{2.5}, \norm{\partial^2_t v}_{1.5}, 
\norm{\partial^3_t v}_0,  \norm{q}_3,  \norm{\partial_t q}_{2},
 \norm{\partial^2_t q}_1 )
   \label{EQ04}
\end{gather}
for a generic polynomial depending on the norms in \eqref{EQ03}.
Finally, we denote by $\ccP_0$ a generic polynomial of 
$\norm{v_0}_3$ and 
$\norm{v_{03}}_{4,\Ga_1}$, i.e.,
\begin{gather}
\ccPz = P(\norm{v_0}_3, \norm{ v_{03}}_{4,\Ga_1}).
\nonumber
\end{gather}
Throughout the paper,  expressions involving the norms 
in \eqref{EQ03} evaluated at time zero are replaced by
$\ccP_0$. That $\ccP_0$ in fact controls all such terms evaluated at time zero is
shown in Section~\ref{section_time_zero};
cf.~\eqref{estimate_time_zero} below.
\label{notation_polynomial}
\end{notation}

\begin{proof}[Proof of Proposition~\ref{proposition_regularity}]
We would like to apply elliptic estimates to \eqref{Lagrangian_bry_q}. While
we do not know a priori that the coefficients $g_{ij}$
 have enough regularity for an application of standard elliptic estimates, we 
can use improved estimates for coefficients with lower regularity
as in \cite{DongKimEllipticBMOHigerOrder}. For this, it suffices to check that
$g_{ij}$ has small oscillation, in the following sense.
Given $ r>0$ and $x \in \Ga_1$, set 
\begin{gather}
\operatorname{osc}_{x,r} (g^{ij}) = 
\frac{1}{\vol(B_r(x))} \int_{B_r(x)} \Big | g^{ij}(y) -
\frac{1}{\vol(B_r(x)) } \int_{B_r(x)} g^{ij} (z)\, dz \Big | \, dy
\nonumber
\end{gather} 
and 
\begin{gather}
g_R = \sup_{x \in \Ga_1} \sup_{r \leq R} \operatorname{osc}_{x,r} (g^{ij}) .
\nonumber
\end{gather}
We need to verify that there exists  $\widetilde{R} \leq 1$ such that 
\begin{gather}
g_{\widetilde{R}}  \leq \rho,
\label{oscillation_condition}
\end{gather}
where $\rho$ is sufficiently small.

Since $g^{ij} \in H^{1.5}(\Ga_1)$, we have  $g^{ij} \in C^{0,\al}(\Ga_1)$ with 
$0< \al < 0.5$ fixed. Thus, for $y \in B_r(x)$,
\begin{align}
\begin{split}
\Big | g^{ij}(y) -
\frac{1}{\vol(B_r(x)) } \int_{B_r(x)} g^{ij} (z)\, dz \Big |
= &
\Big | 
\frac{1}{\vol(B_r(x)) } \int_{B_r(x)} ( g^{ij}(y) -
 g^{ij} (z) )\, dz \Big |
 \\
 \leq &
 \sup_{y,z \in \Ga_1}
 |g^{ij}(y) -
 g^{ij} (z) |
 \leq 
C_{\alpha} r^\al.
\end{split}
\nonumber
\end{align}
Hence,
\begin{gather}
g_{\widetilde{R}} \leq C_{\alpha} \widetilde R^\al,
\nonumber
\end{gather}
and we can ensure \eqref{oscillation_condition}. Therefore, the results of 
\cite{DongKimEllipticBMOHigerOrder} imply that
\begin{align}
\begin{split}
\norm{ \eta_\al }_{3.5,\Ga_1} 
\leq & C( \norm{a_{\mu \al} N_\mu q }_{1.5,\Ga_1}
+ \norm{\eta_{\alpha}}_{1.5,\Ga_1} ) 
\leq
C (\norm{a}_{1.5,\Ga_1} \norm{q}_{1.5,\Ga_1} + \norm{\eta}_{1.5,\Ga_1} ),
\end{split}
\nonumber
\end{align}
where $C$ depends on $\norm{ g_{ij} }_{1.5,\Ga_1}$. Or yet,
\begin{align}
\begin{split}
\norm{ \eta_\al }_{3.5,\Ga_1} 
\leq & C  \norm{ q }_{1.5,\Ga_1}
+ C\norm{\eta}_3 
\leq   C  \norm{q}_{1.5,\Ga_1} + C
\leq  P(\norm{q}_{1.5,\Ga_1}).
\end{split}
\nonumber
\end{align}

We remark that \cite{DongKimEllipticBMOHigerOrder} 
deals only with Sobolev spaces of integer order, but since the estimates
are linear on the norms we can extend them to fractional order Sobolev spaces as well.
\end{proof}

\section{Energy estimate on the three times differentiated system}
\label{section_L_2_estimate} 



In the next statement, we provide two important bounds: on
the $L^2$ norm of $\partial_{t}^{3}v$
and on the $\dot H^{1}$ norm of the projection of 
$\partial_{t}^2v$.
In the last section, we show that these bounds lead to an estimate
for
$\Vert\partial_{t}^2 v_3\Vert_{1,\Gamma_1}$

\begin{lemma}
\label{L01}
We have
\begin{align}
\begin{split}
\norm{ \partial^3_t v }^2_0 
+  \norm{ \overline{\partial} (\Pi \partial^2_t v)}^2_{0,\Ga_1}
\leq &
 \, 
 \widetilde{\epsilon} ( \norm{\partial^2_t v }_{1.5} + \norm{\partial^2_t q}_1 )
   + 
 P(\norm{v}_{2.5+\de}, \norm{q}_{1.5,\Ga_1})
+
\ccP_0 + \int_0^t \ccP,
\end{split}
\label{partial_3_t_v_estimate}
\end{align}
where $\Pi$ is given by \eqref{projection_identity}.
\end{lemma}


We hereafter adopt the following notations.

\begin{notation}
We  use $\widetilde{\epsilon} $ to denote a small positive constant
which may vary from expression to expression. Typically, $\widetilde{\epsilon}$
comes from choosing the time sufficiently small, from Lemma~\ref{lemma_auxiliary_long},
 or from the 
Cauchy inequality with epsilon. The important point to keep in mind, which 
can be easily verified in the expressions containing $\widetilde{\epsilon}$, 
is that once all estimates
are obtained, we can fix $\widetilde{\epsilon}$ 
to be sufficiently small universal constant in order to close the estimates.
\end{notation} 

\begin{notation}
We use $0 < \de < 0.5$ to denote a number that appears in the Sobolev norms (e.g.~$\norm{v}_{2.5+\de}$). 
The values of $\widetilde{\epsilon}$ and
$\de$ may vary from expression to expression, 
but they can be fixed
to hold uniformly across all expressions at the end. 
In Section~\ref{section_eliminating_lower_order} 
we choose $\de$ appropriately.
\end{notation} 



\begin{proof}[Proof of Lemma~\ref{L01}]
We apply $\partial^3_t$ on \eqref{Lagrangian_free_Euler_eq}, contract the resulting equation with $\partial^3_t v_\al$,
and integrate in space and time to find
\begin{align}
\frac{1}{2} \norm{ \partial^3_t v }^2_0 = \frac{1}{2} \norm{\partial^3_t v(0)}^2_0
-\int_0^t \int_\Om \partial^3_t \partial_\mu( a_{\mu \al} q ) \partial^3_t v_\al,
\label{energy_identity_L_2}
\end{align}
where we also used \eqref{div_identity}. Integrating by parts and using that on $\Ga_0$ we have
\begin{align}
\begin{split}
v_3 = 0, \, \text{ so that } \partial_i v_3 = 0 \, \text{ and } \, \partial^3_t v_3 = 0, 
\\
\text{ and } \, \partial_i \eta_3 = 0, \, \text{ so that } \, a_{31} = a_{32} = 0
,
\end{split}
\label{conditions_Ga_0}
\end{align}
we find 
\begin{gather}
 -\int_0^t \int_\Om \partial^3_t \partial_\mu( a_{\mu \al} q ) \partial^3_t v_\al 
= I_1 + I_2,
\label{I_def}
\end{gather}
where 
\begin{align} 
\begin{split}
I_1 
  & = 
  - \int_0^t \int_{\Ga_1} \partial^3_t ( a_{3 \al} q ) \partial^3_t v_\al 
= \si \int_0^t \int_{\Ga_1} \partial^3_t(\sqrt{g} \Delta_g \eta_\al ) \partial^3_t v_\al
\end{split}
\label{I_1_def}
\end{align}
and 
\begin{align}
\begin{split}
I_2 
=  & \, \int_0^t \int_\Om a_{\mu\al} \partial^3_t q \partial_\mu \partial^3_t v_\al
+ 3 \int_0^t \int_\Om \partial_t a_{\mu\al} \partial^2_t q \partial_\mu \partial^3_t v_\al 
\\
& 
+ 3 \int_0^t \int_\Om \partial^2_t a_{\mu \al} \partial_t q \partial_\mu \partial^3_t v_\al
+ \int_0^t \int_\Om \partial^3_t a_{\mu \al} q \partial_\mu \partial^3_t v_\al 
=  I_{21} + I_{22} + I_{23} + I_{24}.
\end{split}
\label{I_2_def}
\end{align}
Note that we used \eqref{aT_identity}  and  \eqref{Lagrangian_bry_q} to rewrite $I_1$.

\emph{Estimate of $I_1$:}
We first use \eqref{partial_Laplacian_eta_identity}
with $\overline{\partial}_A = \partial_t$ and integrate by parts to find
\begin{align}
\begin{split}
\frac{1}{\si} I_1
&=
\int_0^t \int_{\Ga_1} \partial^3_t(\sqrt{g} \Delta_g \eta_\al ) \partial^3_t v_\al 
\\&
 = 
-\int_0^t \int_{\Ga_1}
\sqrt{g} g^{ij} (\de_{\al\la} -g^{kl} \partial_k \eta_\al \partial_l \eta_\la)
\partial^2_t \partial_j v_\la \partial^3_t \partial_i v_\al
\\
&\indeq
-
\int_0^t \int_{\Ga_1}
 \sqrt{g}(g^{ij} g^{kl} - g^{lj}g^{ik} ) \partial_j \eta_\al \partial_k\eta_\la 
\partial^2_t \partial_l v_\la \partial^3_t \partial_i v_\al
\\&\indeq
- 2\int_0^t \int_{\Ga_1}
\partial_t( \sqrt{g} g^{ij} (\de_{\al\la} -g^{kl} \partial_k \eta_\al \partial_l \eta_\la) )
\partial_t \partial_j v_\la \partial^3_t \partial_i v_\al
\\&\indeq
-2
\int_0^t \int_{\Ga_1}
 \partial_t( \sqrt{g}(g^{ij} g^{kl} - g^{lj}g^{ik} ) \partial_j \eta_\al \partial_k\eta_\la )
\partial_t \partial_l v_\la \partial^3_t \partial_i v_\al 
\\&\indeq
- \int_0^t \int_{\Ga_1}
\partial^2_t( \sqrt{g} g^{ij} (\de_{\al\la} -g^{kl} \partial_k \eta_\al \partial_l \eta_\la) )
\partial_j v_\la \partial^3_t \partial_i v_\al
\\&\indeq
-
\int_0^t \int_{\Ga_1}
 \partial^2_t( \sqrt{g}(g^{ij} g^{kl} - g^{lj}g^{ik} ) \partial_j \eta_\al \partial_k\eta_\la )
\partial_l v_\la \partial^3_t \partial_i v_\al 
\\
= &  \, I_{11} + I_{12} + I_{13} + I_{14} + I_{15} + I_{16}. 
\end{split}
\label{I_1_break_up}
\end{align}
The main terms are the first two, 
$I_{11}$ and $I_{12}$; the last four terms are treated similarly
to each other using
integration by parts in time.

We start with
$I_{11}$.
Recalling \eqref{projection_identity}, we have
\begin{align}
\begin{split}
I_{11} & =
- \frac{1}{2}\int_0^t \int_{\Ga_1}
\sqrt{g} g^{ij}\Pi_{\al\la} \partial_t (\partial^2_t \partial_j v_\la \partial^2_t \partial_i v_\al)
\\
& = 
-\frac{1}{2} \int_{\Ga_1} 
\sqrt{g} g^{ij}\Pi_{\al\la} \partial^2_t \partial_j v_\la \partial^2_t \partial_i v_\al
+ I_{11,0} 
 + \frac{1}{2}
\int_0^t \int_{\Ga_1}
\partial_t(\sqrt{g} g^{ij}\Pi_{\al\la} ) \partial^2_t \partial_j v_\la \partial^2_t \partial_i v_\al,
\end{split}
\nonumber
\end{align}
where we used the symmetries of $g^{-1}$ and  $\Pi$ in the first equality, integrated
by parts in time in the second equality, and denoted
\begin{align}
\begin{split}
I_{11,0}
&
=
\left.
\frac{1}{2} \int_{\Ga_1} 
\sqrt{g} g^{ij}\Pi_{\al\la} \partial^2_t \partial_j v_\la \partial^2_t \partial_i v_\al
\right|_{t=0}.
\end{split}
\nonumber
\end{align}
Using \eqref{projections_contraction} 
to separate
$\Pi_{\alpha\lambda}=\Pi_{\alpha\mu}\Pi_{\mu\lambda}$
and writing 
$
\Pi \partial^2_t \partial_i v = \partial_i(\Pi \partial^2_t v) - \partial_i\Pi \partial^2_t v
$,
we obtain
\begin{align}
\begin{split}
I_{11}  = & \,
-\frac{1}{2} \int_{\Ga_1} 
\sqrt{g} g^{ij}\partial_i (\Pi_{\al\mu} \partial^2_t v_\al) \partial_j (\Pi_{\mu\la}  \partial^2_tv_\la)
+ \int_{\Ga_1} \sqrt{g} g^{ij} \partial_i \Pi_{\al\mu} \partial^2_t v_\al \partial_j
(\Pi_{\mu\la} \partial^2_t v_\la )
\\
& -\frac{1}{2} \int_{\Ga_1} \sqrt{g} g^{ij}\partial_i \Pi_{\al\mu} \partial_j \Pi_{\mu\la} \partial^2_t v_\al 
\partial^2_t v_\la 
 + \frac{1}{2}
\int_0^t \int_{\Ga_1}
\partial_t(\sqrt{g} g^{ij}\Pi_{\al\la}) \partial^2_t \partial_j v_\la \partial^2_t \partial_i v_\al 
+  I_{11,0} 
\\
= & \, I_{111} + I_{112} + I_{113} + I_{114} + I_{11,0}.
\end{split}
\nonumber
\end{align}
We proceed to estimate each term. First,
\begin{align}
\begin{split}
 I_{111} & = 
 -\frac{1}{2} \int_{\Ga_1} 
 \de^{ij}\partial_i (\Pi_{\al\mu} \partial^2_t v_\al) \partial_j (\Pi_{\mu\la}  \partial^2_tv_\la)
-\frac{1}{2} \int_{\Ga_1} 
(\sqrt{g} g^{ij} - \de^{ij} )\partial_i (\Pi_{\al\mu} \partial^2_t v_\al) \partial_j (\Pi_{\mu\la}  \partial^2_tv_\la)
\\
& 
= -\frac{1}{2} \norm{ \overline{\partial} (\Pi \partial^2_t v) }^2_{0,\Ga_1}
-\frac{1}{2} \int_{\Ga_1} 
(\sqrt{g} g^{ij} - \de^{ij} )\partial_i (\Pi_{\al\mu} \partial^2_t v_\al) \partial_j (\Pi_{\mu\la}  \partial^2_tv_\la)
\\
& 
\leq
-\frac{1}{2} \norm{ \overline{\partial} (\Pi \partial^2_t v) }^2_{0,\Ga_1}
+ C\norm{ \sqrt{g} g^{-1} - I }_{L^\infty(\Ga_1)} 
\norm{ \overline{\partial} (\Pi \partial^2_t v) }^2_{0,\Ga_1}
\\
&
\leq 
 -\frac{1}{2} \norm{ \overline{\partial} (\Pi \partial^2_t v) }^2_{0,\Ga_1}
+ C\norm{ \sqrt{g} g^{-1} - I }_{1.5,\Ga_1} 
\norm{ \overline{\partial} (\Pi \partial^2_t v) }^2_{0,\Ga_1}.
\end{split}
\nonumber
\end{align}
Writing
$\sqrt{g} g^{-1} - I =
\int_{0}^{t}\partial_{t}(\sqrt{g} g^{-1} - I)$
and estimating the integrand, we realize that if
$T\le 1/C M$, the integral is bounded by $\widetilde{\epsilon}$,
and thus we get
  \begin{equation*}
    I_{111}\ge -\frac14\norm{ \overline{\partial} (\Pi \partial^2_t v) }^2_{0,\Ga_1}   
   .
  \end{equation*}
Next,
\begin{align}
\begin{split}
I_{112} & \leq C \norm{ \sqrt{g} g^{-1} }_{L^\infty(\Ga_1)}
 \norm{ \overline{\partial} \Pi}_{L^\infty(\Ga_1)} \norm{\partial^2_t v}_{0,\Ga_1}
 \norm{\overline{\partial } (\Pi \partial^2_t v)}_{0,\Ga_1}
 \\
 &
 \leq 
 C \norm{ \sqrt{g} g^{-1} }_{1.5,\Ga_1}
 \norm{ \overline{\partial} \Pi}_{1.5,\Ga_1} \norm{\partial^2_t v}_{0,\Ga_1}
 \norm{\overline{\partial } (\Pi \partial^2_t v)}_{0,\Ga_1}
 \\
 &
 \leq P (\norm{q}_{1.5,\Ga_1} ) \norm{ \partial^2_t v}_{0,\Ga_1}
  \norm{\overline{\partial } (\Pi \partial^2_t v)}_{0,\Ga_1}
  \\
  &
  \leq   \widetilde{\epsilon} \norm{\overline{\partial } (\Pi \partial^2_t v)}_{0,\Ga_1}^2
  + C_{\widetilde{\epsilon}}   P(\norm{q}_{1.5,\Ga_1} )
  \norm{\partial^2_t v}_{0,\Ga_1}^2,
\end{split}
\nonumber
\end{align}
where in the third inequality we used Remark~\ref{remark_rational} to get
  \begin{gather}
   \norm{ \sqrt{g} g^{-1} }_{1.5,\Ga_1} \leq C
   \nonumber
  \end{gather}
and 
\begin{gather}
 \norm{ \overline{\partial} \Pi}_{1.5,\Ga_1} \leq C \norm{\overline{\partial} \,\overline{\partial}
 \eta}_{1.5,\Ga_1} \leq C \norm{\eta}_{3.5,\Ga_1}
 \nonumber
\end{gather}
and then invoked Proposition~\ref{proposition_regularity}; also in the fourth line 
we used the $\epsilon$-Cauchy inequality (so $C_{\widetilde{\epsilon}} \rar \infty$
as $\widetilde{\epsilon} \rar 0$).

Again, with the help of Remark~\ref{remark_rational}, we have
\begin{align}
\begin{split}
I_{113}  &\leq C \norm{\overline{\partial} \Pi }_{L^\infty(\Ga_1)}^2 \norm{\partial^2_t v}^2_{0,\Ga_1}
\leq P(\norm{q}_{1.5,\Ga_1}  ) \norm{\partial^2_t v}_{0,\Ga_1}^2
\end{split}
\nonumber
\end{align}
and
\begin{align}
\begin{split}
I_{114} & \leq C \int_0^t \norm{ \partial_t (\sqrt{g} g^{-1} \Pi )}_{L^\infty(\Ga_1)}
\norm{ \overline{\partial} \partial^2_t v}^2_{0,\Ga_1} 
\leq  \int_0^t \ccP,
\end{split}
\nonumber
\end{align}
recalling the notation \eqref{EQ04}.
One also easily obtains
\begin{gather}
I_{11,0} \leq 
P(\norm{ \overline{\partial} \partial^2_t v(0)}_{0,\Ga_1})
=
P(\norm{ \partial^2_t v(0)}_{1,\Ga_1})
\leq\ccPz.
\nonumber
\end{gather}
Hence, choosing $\widetilde{\epsilon}$ sufficiently small 
\begin{align}
\begin{split}
I_{11}  \leq & \, - \frac{1}{C_{11}} \norm{ \overline{\partial} (\Pi \partial^2_t v)}^2_{0,\Ga_1}
+\ccPz
+P(\norm{q}_{1.5,\Ga_1}) \norm{\partial^2_t v}_{0,\Ga_1}^2
+  \int_0^t \ccP,
\end{split}
\label{estimate_I_11_intermediate}
\end{align}
for some positive constant $C_{11}$.
In order to remove the term with $\norm{\partial^2_t v}_{0,\Ga_1}^2$ 
from the right-hand side of \eqref{estimate_I_11_intermediate},
we consider the product
$
\cPt \norm{\partial^2_t v}_{0,\Ga_1}^2$,
where $\cPt = \cPt(t)$ is a polynomial function
in the appropriate norms of $v$ and $q$. Then
by
  \begin{align}
   \begin{split}
      \norm{\partial^2_t v}_{0,\Ga_1}
      \le
      C
      \norm{ \partial^2_t v}_{0}^{1/2}
      \norm{ \partial^2_t v}_{1}^{1/2}
      \leq 
        C 
        \norm{ \partial^2_t v}^{\fractext{2}{3}}_0
        \norm{ \partial^2_t v}^{\fractext{1}{3}}_{1.5}
   \end{split}
   \label{EQ02}
  \end{align}
we get
   \begin{align}
   \begin{split}
     \cPt \norm{ \partial^2_t v}_{0,\Gamma_1}^2 
    &\leq 
     \widetilde{\epsilon} \norm{\partial^2_t v}^2_{1.5} 
      + \cPt^{3} 
      + C  \norm{\partial^2_t v}^4_0
    \\&
    \leq
     \widetilde{\epsilon} \norm{\partial^2_t v}^2_{1.5} 
      + \cPt^{3} 
      + C  \norm{\partial^2_t v_0}^4_0
      + C 
        \left(
          \int_{0}^{t}
           \norm{\partial^3_t v}_0
        \right)^{3}
    \\&
    \leq
     \widetilde{\epsilon} \norm{\partial^2_t v}^2_{1.5} 
      + \cPt^{3} 
      + C  \norm{\partial^2_t v_0}^4_0
      + C T^{2}
          \int_{0}^{t}
           \norm{\partial^3_t v}_0^{3}
   .
   \end{split}
   \nonumber
   \end{align}
By $T\le C$, we conclude
\begin{align}
\begin{split}
\cPt \norm{ \partial^2_t v}_{1}^2 \leq   & \,
\widetilde{\epsilon} \norm{\partial^2_t v}_{1.5}^2 
+ \cPt^3 
+\ccPz
+ \int_0^t \ccP.
\end{split}
\label{estimate_partial_2_t_v_lower_order}
\end{align}
Since it is needed below, we note that similar arguments give
\begin{gather}
\cPt \norm{\partial_t v}^2_{1.5} \leq \widetilde{\epsilon} \norm{\partial_t v}_{2.5}^2 
+ \cPt^a 
+ \ccPz
+   \int_0^t \ccP,
\label{estimate_partial_t_v_lower_order}
\end{gather}
where $a  \geq 1$. We also observe that since 
the power $a$ comes from an application of the 
Young inequality (e.g.~$a=3$ in
\eqref{estimate_partial_2_t_v_lower_order}
above), 
we can 
choose  it so that $\cPt^a$ is a polynomial if $\cPt$ is. Finally, we remark that 
we can obtain similar estimates for other lower order norms  of $v$ and $q$, (e.g.~$\norm{\partial_t v}_2$ or $\norm{\partial_t q}_1$).

Using \eqref{estimate_partial_2_t_v_lower_order}, 
the inequality \eqref{estimate_I_11_intermediate} becomes
\begin{align}
\begin{split}
I_{11}  \leq  & \, - \frac{1}{C_{11}} \norm{ \overline{\partial} (\Pi \partial^2_t v)}^2_{0,\Ga_1}
+ \widetilde{\epsilon} \norm{ \partial^2_t v}_{1.5}^2 
+P(\norm{  \partial^2_t v(0)}_{1.5})
+P(\norm{q}_{1.5,\Ga_1}) 
+   \int_0^t \ccP
.
\end{split}
\label{estimate_I_11}
\end{align}


The term $I_{12}$ is estimated with a trick used 
in 
\cite[p.~868]{CoutandShkollerFreeBoundary}, where the authors observed that 
we may write
\begin{align}
\begin{split}
I_{12} & = \int_0^t \int_{\Gamma_1}\frac{1}{\sqrt{g}} (\partial_t \det A^1 - \det A^2 - \det A^3 ),
\end{split}
\label{I_12_trick}
\end{align}
where
\begin{align}
\begin{split}
A^1 = 
\begin{bmatrix}
\partial_1 \eta_\mu \partial^2_t \partial_1 v_\mu
&
\partial_1 \eta_\mu \partial^2_t \partial_2 v_\mu
\\
\partial_2 \eta_\mu \partial^2_t \partial_1 v_\mu
&
\partial_2 \eta_\mu \partial^2_t \partial_2 v_\mu
\end{bmatrix},
A^2 = 
\begin{bmatrix}
\partial_1 v_\mu \partial^2_t \partial_1 v_\mu
&
\partial_1 \eta_\mu \partial^2_t \partial_2 v_\mu
\\
\partial_2 v_\mu \partial^2_t \partial_1 v_\mu
&
\partial_2 \eta_\mu \partial^2_t \partial_2 v_\mu
\end{bmatrix},
A^3 = 
\begin{bmatrix}
\partial_1 \eta_\mu \partial^2_t \partial_1 v_\mu
&
\partial_1 v_\mu \partial^2_t \partial_2 v_\mu
\\
\partial_2 \eta_\mu \partial^2_t \partial_1 v_\mu
&
\partial_2 v_\mu \partial^2_t \partial_2 v_\mu
\end{bmatrix}.
\end{split}
\nonumber
\end{align}
Integrating by parts in time in the first term in \eqref{I_12_trick},
we get
\begin{align}
\begin{split}
I_{12} & = \int_{\Ga_1} \frac{1}{\sqrt{g}} \det A^1 
- \left.  \int_{\Ga_1} \frac{1}{\sqrt{g}} \det A^1 \right|_{t=0}
- \int_0^t \int_{\Ga_1} \partial_t \left(\frac{1}{\sqrt{g}}\right) \det A^1 
- \int_0^t \int_{\Ga_1} \frac{1}{\sqrt{g}}( \det A^2 + \det A^3)
\\
&
= I_{121} + I_{12,0} + I_{122} + I_{123}.
\end{split}
\label{I_12_break_up}
\end{align}
First, 
\begin{align}
\begin{split}
\int_{\Ga_1} \frac{1}{\sqrt{g}} \det A^1
& = \int_{\Ga_1} Q^i_{\mu \la}
              (\overline{\partial}{} \eta,\overline{\partial}{}^2 \eta) \partial^2_t v_\mu \partial_i \partial^2_t 
v_\la
.
\end{split}
\label{calculation_det}
\end{align}
To see why \eqref{calculation_det} holds, we compute the determinant to find
\begin{align}
\begin{split}
 \int_{\Ga_1} \frac{1}{\sqrt{g}} \det A^1 = &
  \int_{\Ga_1} \frac{1}{\sqrt{g}} 
( \partial_1 \eta_\mu  
\partial_2 \eta_\la 
\partial_1\partial^2_t v_\mu 
  \partial_2 \partial^2_t  v_\la
-
\partial_1 \eta_\mu
\partial_2 \eta_\la  
\partial_2 \partial^2_t  v_\mu
 \partial_1  \partial^2_t  v_\la
).
\end{split}
\nonumber
\end{align}
Integrating by 
parts the $\partial_1$ derivative 
in the factor $\partial_1\partial^2_t v_\mu $, and 
the $\partial_2$ derivative in 
$\partial_2 \partial^2_t  v_\mu$ produces \eqref{calculation_det} since
the terms with four derivatives cancel out.
Thus we obtain
\begin{align}
\begin{split}
I_{121} 
& 
\leq  P( \norm{\overline{\partial}{}^2 \eta}_{L^\infty(\Ga_1)} ) 
\norm{\partial^2_t v}_{0,\Ga_1} \norm{ \overline{\partial} \partial^2_t v}_{0,\Ga_1}
\leq   P( \norm{\eta}_{3.5,\Ga_1} ) 
\norm{\partial^2_t v}_{0,\Ga_1} \norm{ \overline{\partial} \partial^2_t v}_{0,\Ga_1}
 \\
 & \leq 
  \widetilde{\epsilon}  \norm{ \overline{\partial} \partial^2_t v}_{0,\Ga_1}^2
 +  
    C_{\widetilde{\epsilon}}
    P(\norm{q}_{1.5,\Ga_1} )
    \norm{\partial^2_t v}_{0,\Ga_1}^2,
\end{split}
\nonumber
\end{align}
where we used Proposition~\ref{proposition_regularity}. 
Next,
\begin{align}
\begin{split}
I_{122} & = 
- \int_0^t \int_{\Ga_1} \partial_t \left(\frac{1}{\sqrt{g}}\right)
(\partial_1 \eta_\mu \partial^2_t \partial_1 v_\mu
\partial_2 \eta_\la \partial^2_t \partial_2 v_\la
- 
\partial_2 \eta_\mu \partial^2_t \partial_1 v_\mu
\partial_1 \eta_\la \partial^2_t \partial_2 v_\la )
\\
& 
\leq C \int_0^t \nnorm{\partial_t \left(\frac{1}{\sqrt{g}}\right) }_{L^\infty(\Ga_1)}
\norm{ \overline{\partial} \eta }_{L^\infty(\Ga_1)}^2
\norm{ \overline{\partial} \partial^2_t v }^2_{0,\Ga_1}
\leq
\int_{0}^{t}\ccP,
\end{split}
\nonumber
\end{align}
where we used \eqref{rational_estimate}
to estimate
\begin{gather}
\nnorm{\partial_t \left(\frac{1}{\sqrt{g}}\right) }_{1.5,\Ga_1} \leq 
C
\norm{\partial_t \overline{\partial} \eta }_{1.5,\Ga_1}
\leq 
C\norm{v}_3.
\nonumber
\end{gather}
Finally,
\begin{align}
\begin{split}
\int_0^t \int_{\Ga_1} \frac{1}{\sqrt{g}} \det A^2
& = 
\int_0^t \int_{\Ga_1} \frac{1}{\sqrt{g}}
( \partial_1 v_\mu \partial^2_t \partial_1 v_\mu
\partial_2 \eta_\la \partial^2_t \partial_2 v_\la
-
\partial_2 v_\mu \partial^2_t \partial_1 v_\mu 
\partial_1 \eta_\la \partial^2_t \partial_2 v_\la )
\\
&
\leq C \int_0^t \nnorm{\frac{1}{\sqrt{g}} }_{L^\infty(\Ga_1)}
\norm{ \overline{\partial} v }_{L^\infty(\Ga_1)}^2
\norm{ \overline{\partial} \partial^2_t v }^2_{0,\Ga_1}
\le
\int_{0}^{t}\ccP
\end{split}
\nonumber
\end{align}
by \eqref{rational_estimate_2}. The term with $\det A^3$ is estimated
the same way,  and thus
$
I_{123} \leq  \int_0^t \ccP.
$
Since also
$
I_{12,0} \leq P(\norm{\partial^2_t v(0)}_{1.5})
$,
we conclude that
\begin{align}
\begin{split}
I_{12} \leq &
 \widetilde{\epsilon}  \norm{ \overline{\partial} \partial^2_t v}_{0,\Ga_1}^2 +
 \widetilde{\epsilon} \norm{\partial^2_t v}_{1.5}^2 
 + \ccPz
 +  P(\norm{q}_{1.5,\Ga_1} )
 + \int_0^t \ccP,
 \end{split}
\label{estimate_I_12}
\end{align}
where we used \eqref{estimate_partial_2_t_v_lower_order}.

The terms $I_{13}$ are $I_{14}$ are estimated analogously, so we only
show the details for $I_{13}$.
To treat $I_{13}$ (and $I_{14}$), it suffices
to keep track of the general multiplicative structure of the integrands, and for this
we use the symbolic notation in Remark~\ref{remark_rational} to write
\begin{gather}
I_{13}
= - \int_0^t \int_{\Ga_1} \partial_t Q(\overline{\partial} \eta)
\partial_t \overline{\partial} v \partial^3_t \overline{\partial} v.
\nonumber
\end{gather}
From \eqref{partial_Q_symb} we have
\begin{gather}
\partial^2_t Q(\overline{\partial} \eta) = Q( \overline{\partial} \eta) 
(\partial_t \overline{\partial} \eta)^2 + 
Q( \overline{\partial} \eta) 
\partial^2_t \overline{\partial} \eta
=
Q( \overline{\partial} \eta) 
(\overline{\partial} v)^2 + 
Q( \overline{\partial} \eta) 
\partial_t \overline{\partial} v
,
\label{partial_2_t_Q_symb}
\end{gather}
so that integration by parts in $t$ gives
\begin{align}
\begin{split}
I_{13} = & \, 
-\int_{\Ga_1} \partial_t Q(\overline{\partial}\eta) \partial_t \overline{\partial}v \partial^2_t \overline{\partial}v 
+ I_{13,0}
\\
&
+ \int_0^t \int_{\Ga_1} \partial^2_t Q(\overline{\partial} \eta) \partial_t 
\overline{\partial} v \partial^2_t \overline{\partial} v
+\int_0^t \int_{\Ga_1} \partial_t Q(\overline{\partial} \eta)
\partial^2_t \overline{\partial} v \partial^2_t \overline{\partial} v,
\end{split}
\nonumber
\end{align}
where 
\begin{gather}
I_{13,0} =
\left. 
\int_{\Ga_1} \partial_t Q(\overline{\partial}\eta) \partial_t \overline{\partial}v \partial^2_t \overline{\partial}v 
\right|_{t=0}\leq \ccPz
.
\nonumber
\end{gather}
Thus, recalling \eqref{rational_estimate} and using \eqref{partial_2_t_Q_symb},
\begin{align}
\begin{split}
I_{13} \leq & \, 
C \norm{ \partial_t \overline{\partial} \eta}_{L^\infty(\Ga_1)} 
\norm{\partial_t \overline{\partial} v}_{0,\Ga_1} 
\norm{\partial^2_t \overline{\partial} v}_{0,\Ga_1} 
+ 
\ccPz
+ C \int_0^t  \norm{ \overline{\partial} v}^2_{L^\infty(\Ga_1)} 
\norm{\partial_t \overline{\partial} v}_{0,\Ga_1} 
\norm{\partial^2_t \overline{\partial} v}_{0,\Ga_1}
\\&
+
C\int_0^t \norm{\partial_t \overline{\partial} v}^2_{L^4(\Ga_1)}
\norm{\partial^2_t \overline{\partial} v}_{0,\Ga_1}
+ C\int_0^t \norm{\partial_t \overline{\partial} \eta}_{L^\infty(\Ga_1)} 
\norm{\partial^2_t \overline{\partial} v}_{0,\Ga_1}^2
\\
\leq & \, \widetilde{\epsilon}  \norm{\partial^2_t v}_{1.5}^2
+ 
\ccPz
+ C_{\widetilde{\epsilon}} \norm{\partial_t v}_{1.5}^2\norm{v}^2_{2.5+\de}
+ 
\int_0^t \ccP.
\end{split}
\label{estimate_I_13_14_intermediate}
\end{align}
It follows that the inequality~\eqref{estimate_I_13_14_intermediate} becomes
\begin{align}
\begin{split}
I_{13}%
\leq  & \, 
\, \widetilde{\epsilon}  \norm{\partial^2_t v}_{1.5}^2 
+\ccPz
+ P(\norm{v}_{2.5+\de})
+ 
C\int_0^t (1 + \norm{\partial_t v}_{1.5} ) \norm{\partial_t v}_{1.5} 
\norm{\partial^2_t v}_{1.5} 
+ \int_0^t \ccP.
\end{split}
\label{estimate_I_13_14}
\end{align}
As pointed out above, the same inequality holds for $I_{14}$ as well.

The remaining terms in $I_{1}$, which are 
$I_{15}$ and $I_{16}$, are estimated analogously, so we only
estimate~$I_{15}$.
It again suffices
to keep track of the general multiplicative structure of the integrands and
we use the symbolic notation of Remark~\ref{remark_rational}.
From \eqref{partial_Q_symb} and  \eqref{partial_2_t_Q_symb},
\begin{gather}
\partial^3_t Q(\overline{\partial} \eta) = Q( \overline{\partial} \eta) 
(\partial_t \overline{\partial} \eta)^3 + 
Q(\overline{\partial} \eta) \partial_t \overline{\partial}\eta
\partial^2_t \overline{\partial} \eta 
+ Q( \overline{\partial} \eta) 
\partial^3_t \overline{\partial} \eta
\label{partial_3_t_Q_symb}
\end{gather}
and thus, using integration by parts in time,
\begin{align}
\begin{split}
I_{15}  = & 
I_{15,0}
- \int_{\Ga_1} \partial^2_t Q(\overline{\partial} \eta) \overline{\partial} v
\partial^2_t \overline{\partial} v 
+ \int_0^t \int_{\Ga_1} \partial^3_t Q(\overline{\partial} \eta)
\overline{\partial}v \partial^2_t \overline{\partial} v
+ \int_0^t \int_{\Ga_1} \partial^2_t Q(\overline{\partial} \eta) \partial_t \overline{\partial} v \partial^2_t \overline{\partial} v
\\
= & \, 
I_{15,0}
- \int_{\Ga_1}Q(\overline{\partial} \eta) 
(\partial_t \overline{\partial} \eta)^2
\overline{\partial} v
\partial^2_t \overline{\partial} v  
- \int_{\Ga_1}Q(\overline{\partial} \eta) 
\partial^2_t \overline{\partial} \eta
\overline{\partial} v
\partial^2_t \overline{\partial} v
+ \int_0^t \int_{\Ga_1} Q(\overline{\partial} \eta)  (\partial_t \overline{\partial} \eta)^3
\overline{\partial} v
\partial^2_t \overline{\partial} v
\\
&
+ \int_0^t \int_{\Ga_1} Q(\overline{\partial} \eta) \partial_t \overline{\partial}\eta
\partial^2_t \overline{\partial} \eta
\overline{\partial} v
\partial^2_t \overline{\partial} v
+ \int_0^t \int_{\Ga_1} Q( \overline{\partial} \eta) 
\partial^3_t \overline{\partial} \eta
\overline{\partial} v
\partial^2_t \overline{\partial} v
+
\int_0^t \int_{\Ga_1} \partial^2_t Q(\overline{\partial} \eta) \partial_t \overline{\partial} v \partial^2_t \overline{\partial} v
\\ &\!\!\!
=
I_{15,0}
+ I_{151}
+ I_{152}
+ I_{153}
+ I_{154}
+ I_{155}
+ I_{156},
\end{split}
\nonumber
\end{align}
where we 
set
\begin{gather}
I_{15,0}
=
\left. \int_{\Ga_1} \partial^2_t Q(\overline{\partial} \eta) \overline{\partial} v
\partial^2_t \overline{\partial} v \right|_{t=0}
\leq \ccPz
   .%
\nonumber
\end{gather}
We have
\begin{align}
\begin{split}
I_{151}  & \leq 
C \norm{\overline{\partial} v}^3_{L^\infty(\Ga_1)} 
\norm{\partial^2_t \overline{\partial} v}_{0,\Ga_1} 
\leq \widetilde{\epsilon} \norm{\partial^2_t v}_{1.5}^2 + 
P(\norm{v}_{2.5+\de}),
\end{split}
\nonumber
\end{align}
as well as
\begin{align}
\begin{split}
I_{152}  & \leq 
C \norm{\overline{\partial} \partial_t v}_{0,\Ga_1} 
\norm{\overline{\partial} \partial^2_t v}_{0,\Ga_1} 
\norm{\overline{\partial} v}_{L^\infty(\Ga_1)}
\leq 
\widetilde{\epsilon} \norm{\partial^2_t v}^2_{1.5} 
+ C\norm{\partial_t v}^2_{1.5} 
+ P(\norm{v}_{2.5+\de}),
\end{split}
\nonumber
\end{align}
and
\begin{align}
\begin{split}
I_{153} + I_{154} + I_{155}  & \leq 
\int_0^t\ccP.
\end{split}
\nonumber
\end{align}
The term  $I_{156} $ has already been dealt with in the estimate
of $I_{13}$ and obeys \eqref{estimate_I_13_14}. Thus, from the above we obtain
\begin{align}
\begin{split}
I_{15} \leq &
\,
\widetilde{\epsilon} \norm{\partial^2_t v}^2_{1.5}  
+ \ccPz
+  P(\norm{v}_{2.5+\de})
+ C\norm{\partial_t v}^2_{1.5}  +  \int_0^t \ccP.
\end{split}
\label{estimate_I_14_15_intermediate}
\end{align}
Invoking \eqref{estimate_partial_t_v_lower_order} and \eqref{estimate_I_14_15_intermediate} gives
\begin{align}
\begin{split}
I_{15} + I_{16}  \leq &
\,
\widetilde{\epsilon} \norm{\partial^2_t v}^2_{1.5}  
+ \ccPz
+  P(\norm{v}_{2.5+\de})
 +  \int_0^t \ccP.
\end{split}
\label{estimate_I_14_15}
\end{align}

Combining \eqref{estimate_I_11}, \eqref{estimate_I_12}, \eqref{estimate_I_13_14},
and \eqref{estimate_I_14_15}, and recalling \eqref{I_1_def} and \eqref{I_1_break_up},
we obtain
\begin{align}
\begin{split}
I_1  \leq  & \, 
- \frac{1}{C_{11}} \norm{ \overline{\partial} (\Pi \partial^2_t v)}^2_{0,\Ga_1}
+
\widetilde{\epsilon}  \norm{ \partial^2_t v}_{1.5}^2
+ P(\norm{q}_{1.5,\Ga_1}, \norm{v}_{2.5+\de} )
+\ccPz
+
  \int_0^t \ccP.\end{split}
\label{estimate_I_1}
\end{align}


\emph{Estimate of $I_2$:}
Now, we consider $I_2$, rewritten in \eqref{I_2_def}. 
%
Differentiating \eqref{Lagrangian_free_Euler_div} in time three times gives
\begin{align}
\begin{split}
a_{\mu \al} \partial_\mu \partial^3_t v_\al 
& =
-3 \partial^2_t a_{\mu \al} \partial_\mu \partial_t v_\al
-3 \partial_t a_{\mu \al} \partial_\mu \partial^2_t v_\al 
- \partial^3_t a_{\mu \al} \partial_\mu v_\al,
\end{split}
\nonumber
\end{align}
so that $I_{21}$ becomes
\begin{align}
\begin{split}
I_{21}  &= 
-3 \int_0^t \int_\Om \partial^2_t a_{\mu \al} \partial_\mu \partial_t v_\al \partial^3_t q
- 3\int_0^t \int_\Om \partial_t a_{\mu \al} \partial_\mu \partial^2_t v_\al 
\partial^3_t q
- \int_0^t \int_\Om \partial^3_t a_{\mu \al} \partial_\mu v_\al\partial^3_t q
\\
& 
= I_{211} + I_{212} + I_{213}.
\end{split}
\label{I_21_break_up}
\end{align}
We have
\begin{align}
\begin{split}
I_{211}  = &
 \left. -3 \int_\Om 
\partial^2_t a_{\mu \al} \partial_\mu \partial_t v_\al \partial^2_t q
\right|_0^t 
+ 3 \int_0^t \int_\Om
\partial_t ( 
\partial^2_t a_{\mu \al} \partial_\mu \partial_t v_\al )
\partial^2_t q
\\
\leq  &
C \left. \norm{\partial^2_t a}_0 \norm{\partial \partial_t v}_{L^3(\Om)} 
\norm{\partial^2_t q}_{L^6(\Om)} \right|_0^t
+ C\int_0^t \norm{\partial^3_t a}_{L^3(\Om)}  \norm{\partial \partial_t v}_{L^6(\Om)}
\norm{\partial^2_t q}_{L^2(\Om)}
\\
&
+ C\int_0^t \norm{\partial^2_t a}_{L^6(\Om)} \norm{\partial \partial^2_t v}_{L^3(\Om)}
\norm{\partial^2_t q}_{L^2(\Om)}
\\
 \leq & 
 C \left. (\norm{\partial v}_0 \norm{\partial v}_{L^\infty(\Om)} + \norm{\partial \partial_t v}_0 )
\norm{\partial \partial_t v}_{0.5} \norm{\partial^2_t q}_1 \right|_0^t
\\
& 
+
C\int_0^t (\norm{\partial v}_{L^3(\Om)} \norm{\partial v}^2_{L^\infty(\Om)}
+ \norm{\partial \partial_t v}_{L^3(\Om)} \norm{\partial v}_{L^\infty(\Om)} 
+\norm{\partial \partial^2_t v}_{L^3(\Om)} ) \norm{\partial_t v}_2 \norm{\partial^2_t q}_1
\\
& +
C\int_0^t (\norm{\partial v}_1 \norm{\partial v}_{L^\infty(\Om)} + \norm{\partial \partial_t v}_1 )
\norm{\partial_t^2 v}_{1.5} \norm{\partial^2_t q}_{1}
\\
\leq & \, \widetilde{\epsilon}(  \norm{\partial^2_t q}_1^2 + \norm{\partial_t v}^2_{2.5} )+  P( \norm{v}_{2.5+\de}) 
+\ccPz
+ \int_{0}^{t}\ccP,
\end{split}
\label{estimate_I_211}
\end{align}
where in the last step we used an argument similar to that  for 
\eqref{estimate_partial_t_v_lower_order}.
Next, we integrate by parts and use \eqref{div_identity} to find
\begin{align}
\begin{split}
I_{212} & = 
-3 \int_0^t \int_{\partial \Om} \partial_t a_{3 \al} \partial^2_t v_\al \partial^3_t q N_3
+ 3 \int_0^t \int_\Om \partial_t a_{\mu \al} \partial^2_t v_\al \partial_\mu \partial^3_t q
\\& 
= I_{2121} + I_{2122}.
\end{split}
\label{I_212_break_up}
\end{align}
By \eqref{Lagrangian_free_Euler_a_eq} and \eqref{Lagrangian_bry_v}, we have
\begin{align}
\begin{split}
I_{2121} & = 
3 \int_0^t \int_{\Ga_1} a_{3\ga} \partial_\nu v_\ga a_{\nu \al }\partial^2_t v_\al \partial^3_t q.
\end{split}
\label{I_2121_def}
\end{align}
Since in light of \eqref{Lagrangian_bry_q} and \eqref{aT_identity},
\begin{align}
\begin{split}
a_{3\ga}  \partial^3_t q
& = - \si \partial^3_t(\sqrt{g} \Delta_g \eta_\ga) 
-3 \partial_t a_{3 \ga} \partial^2_t q
-3 \partial^2_t a_{3 \ga} \partial_t q
- \partial^3_t a_{3 \ga} q,
\end{split}
\nonumber
\end{align}
we obtain
\begin{align}
\begin{split}
I_{2121}  = & 
- 3\si 
\int_0^t \int_{\Ga_1}  \partial_\nu v_\ga a_{\nu \al} \partial^2_t v_\al
\partial^3_t(\sqrt{g} \Delta_g \eta_\ga)
-9 \int_0^t \int_{\Ga_1}  \partial_\nu v_\ga a_{\nu \al} \partial^2_t v_\al \partial_t a_{3 \ga} \partial^2_t q
\\
&
-9 \int_0^t \int_{\Ga_1}  \partial_\nu v_\ga a_{\nu \al} \partial^2_t v_\al \partial^2_t a_{3 \ga} \partial_t q
-3\int_0^t \int_{\Ga_1}  \partial_\nu v_\ga a_{\nu \al} \partial^2_t v_\al \partial^3_t a_{3 \ga}q
\\
= & \,
I_{21211} + 
I_{21212} + 
I_{21213} + 
I_{21214}. 
\end{split}
\label{I_2121_break_up}
\end{align}
In order to estimate $I_{21211}$, 
we use \eqref{partial_Laplacian_eta_identity} 
and $\partial_{t}\eta=v$
to write
\begin{align}
\begin{split}
I_{21211}
 = & 
- 3\si \int_0^t \int_{\Ga_1}
\partial_i \big (
\sqrt{g} g^{ij} (\de_{\ga\la} -g^{kl} \partial_k \eta_\ga \partial_l \eta_\la)
\partial^2_t \partial_j v_\la 
\big )
\partial_\nu v_\ga a_{\nu \al} \partial^2_t v_\al
\\
&
-3\si 
\int_0^t \int_{\Ga_1}
\partial_i \big (
 \sqrt{g}(g^{ij} g^{kl} - g^{lj}g^{ik} ) \partial_j \eta_\ga \partial_k\eta_\la 
\partial^2_t \partial_l v_\la 
\big )
\partial_\nu v_\ga a_{\nu \al} \partial^2_t v_\al
\\
&
- 6\si \int_0^t \int_{\Ga_1}
\partial_i \big (
\partial_t( \sqrt{g} g^{ij} (\de_{\ga\la} -g^{kl} \partial_k \eta_\ga \partial_l \eta_\la) )
\partial_t \partial_j v_\la 
\big )
\partial_\nu v_\ga a_{\nu \al} \partial^2_t v_\al
\\
&
-6\si 
\int_0^t \int_{\Ga_1}
\partial_i \big (
 \partial_t( \sqrt{g}(g^{ij} g^{kl} - g^{lj}g^{ik} ) \partial_j \eta_\ga \partial_k\eta_\la )
\partial_t \partial_l v_\la 
\big )
\partial_\nu v_\ga a_{\nu \al} \partial^2_t v_\al
\\
&
- 3\si \int_0^t \int_{\Ga_1}
\partial_i \big (
\partial^2_t( \sqrt{g} g^{ij} (\de_{\ga\la} -g^{kl} \partial_k \eta_\ga \partial_l \eta_\la) )
\partial_j v_\la 
\big )
\partial_\nu v_\ga a_{\nu \al} \partial^2_t v_\al
\\
&
-3\si 
\int_0^t \int_{\Ga_1}
\partial_i \big (
 \partial^2_t( \sqrt{g}(g^{ij} g^{kl} - g^{lj}g^{ik} ) \partial_j \eta_\ga \partial_k\eta_\la )
\partial_l v_\la 
\big )
\partial_\nu v_\ga a_{\nu \al} \partial^2_t v_\al.
\end{split}
\nonumber
\end{align}
In each of the above integrals, we integrate by parts the derivative $\partial_i$.
After this, the sum of the first two integrals has the form
$3 \si \int_0^t \int_{\Ga_1} Q(\overline{\partial}\eta ) \partial^2_t \overline{\partial} v
\overline{\partial}( \overline{\partial} v a \partial^2_t v )
$ (cf.~Remark~\ref{remark_rational}),
the sum of the middle terms is
of the form
$6 \si
\int_0^t \int_{\Ga_1} 
\partial_t Q(\overline{\partial} \eta) \partial_t \overline{\partial} v
\overline{\partial}( \overline{\partial} v a \partial^2_t v )
$,
while the sum of the last two is
$3 \si
\int_0^t \int_{\Ga_1}
\partial^2_t Q(\overline{\partial} \eta ) \overline{\partial} v
\overline{\partial}( \overline{\partial} v a \partial^2_t v )
$.
Each of the three forms is bounded by
$\int_{0}^{t}\ccP$, using H\"older and Sobolev inequalities, and we obtain
\begin{align}
\begin{split}
I_{21211} \leq & 
\int_0^t \ccP.
\end{split}
\label{estimate_I_21211}
\end{align}
The terms $I_{21212}$ and $I_{21213}$ are also bounded using
H\"older and Sobolev inequalities, leading to

\begin{align}
\begin{split}
I_{21212},
I_{21213}\le 
\int_0^t  \ccP
   .
\end{split}
\notag
\end{align}
For instance, we have
\begin{align}
\begin{split}
I_{21213}
=&
-9 \int_0^t \int_{\Ga_1}  \partial_\nu v_\ga a_{\nu \al} \partial^2_t v_\al \partial^2_t a_{3 \ga} \partial_t q
\\
\leq &
C
\int_0^t
 \norm{ \partial v}_{L^6(\Ga_1)}
 \norm{ a}_{L^\infty(\Ga_1)}
  \norm{ \partial^2_t v}_{L^6(\Ga_1)}
 \norm{ \partial^2_t a}_{0,\Ga_1}
  \norm{ \partial_t q}_{L^6(\Ga_1)}   
  \leq 
\int_0^t  \ccP.
\end{split}
\label{estimate_I_21213}
\end{align}
%
%
The term $I_{21214}$ can not be estimated immediately
due to the factor $\partial_{t}^{3}a$.
First, we write
\begin{align}
\begin{split}
I_{21214} = & 
-3\int_0^t \int_{\Ga_1}  \partial_\nu v_\ga a_{\nu \al} \partial^2_t v_\al \partial^3_t a_{3 \ga}  q
=
-3\int_0^t \int_{\Ga_1}  
\partial_\nu v_\ga 
a_{\nu \al} 
\partial^2_t v_\al 
\partial^3_t a_{3 \ga}
 q
\\
\leq &
C \sum_{\ga=1}^3  \int_0^t 
\norm{ \partial v}_{L^\infty(\Ga_1)}
\norm{ a}_{L^\infty(\Ga_1)}
\norm{ \partial^2_t v}_{0,\Ga_1}
\norm{ \partial^3_t a_{3 \ga}}_{0,\Ga_1}
\norm{ q}_{L^\infty(\Ga_1)}
\\
\leq & 
C\sum_{\ga=1}^3 
\int_0^t
\norm{v}_3 
\norm{ \partial^2_t v}_{1.5} 
\norm{ \partial^3_t a_{3 \ga}}_{0,\Ga_1}
\norm{q}_3.
\end{split}
\nonumber
\end{align}
To bound $\partial^3_t a_{3 \ga}$, we use that $a_{3\ga}$ contains
only tangential derivatives of $\eta$, i.e., from \eqref{a_explicit} we see that
\begin{gather}
a_{3\ga} = \epsilon_{\ga \al \be} \partial_1 \eta_\al \partial_2 \eta_\be,
\label{a_3_ga_bry}
\end{gather}
where $\epsilon_{\ga \al \be}$ is the Levi-Civita symbol,
so that 
\begin{align}
\begin{split}
\norm{\partial^3_t a_{3\ga} }_{0,\Ga_1} 
\leq & C \norm{ \overline{\partial} \partial^2_t v \overline{\partial} \eta}_{0,\Ga_1}
+
C \norm{ \overline{\partial} \partial_t v \overline{\partial} v}_{0,\Ga_1}
\leq
P (\norm{\partial^2_t v}_{1.5}, \norm{\partial_t v}_{1.5},\norm{v}_3 ).
\end{split}
\nonumber
\end{align}
Hence,
$
I_{21214} \leq 
\int_0^t \ccP
$, and 
we obtain
\begin{align}
\begin{split}
I_{2121} \leq  \int_0^t \ccP
   .
\end{split}
\label{estimate_I_2121}
\end{align}
The term $I_{2122}$ is estimated 
using integration by parts in $t$
and obeys
\begin{align}
\begin{split}
I_{2122} \leq & 
C \norm{\partial^2_t q}_1 \norm{\partial^2_t v}_{L^3(\Omega)}
\norm{\partial v}_{L^6(\Om)}
+
P(
\norm{\partial^2_t q(0)}_1, \norm{\partial^2_t v(0)}_{1.5}, \norm{v(0)}_3)
+ \int_0^t \ccP 
\\
\leq 
&
C \norm{\partial^2_t q}_1 \norm{\partial^2_t v}_{0.5} \norm{v}_2
+
P(
\norm{\partial^2_t q(0)}_1, \norm{\partial^2_t v(0)}_{1.5}, \norm{v(0)}_3)
+ \int_0^t \ccP
\\
\leq &
\,
\widetilde{\epsilon} ( \norm{\partial^2_t q}^2_1 
+ \norm{\partial^2_t v}^2_{1.5} ) 
+
\ccPz
+ P(\norm{v}_{2}) 
+ \int_0^t \ccP,
\end{split}
\label{estimate_I_2122}
\end{align}
where in the last line we estimated similarly to \eqref{estimate_partial_2_t_v_lower_order} and 
\eqref{estimate_partial_t_v_lower_order}.

Combining \eqref{I_212_break_up}, \eqref{estimate_I_2121},
and \eqref{estimate_I_2122}, we obtain 
\begin{align}
\begin{split}
I_{212} \leq &
\,
\widetilde{\epsilon}( \norm{\partial^2_t q}^2_1 
+ \norm{\partial^2_t v}^2_{1.5} )
+
\ccPz
+ P(\norm{v}_{2}) 
+ \int_0^t \ccP.
\end{split}
\label{estimate_I_212} 
\end{align}
Next we estimate $I_{213}$. Using \eqref{Lagrangian_free_Euler_a_eq}, we find
\begin{align}
\begin{split}
  I_{213}
=
- \int_0^t \int_\Om \partial^3_t a_{\mu \al} \partial_\mu v_\al\partial^3_t q
= & 
\int_0^t \int_\Om   a_{\mu \ga} \partial^2_t \partial_\be v_\ga a_{\be \al}  \partial_\mu v_\al\partial^3_t q
+ R,
\end{split}
\nonumber
\end{align}
where $R$ 
represents the sum of the lower order terms, which
are
easily
estimated as
\begin{align}
\begin{split}
R \leq & 
\,
\widetilde{\epsilon} (\norm{ \partial^2_t q}^2_1 + \norm{ \partial^2_t v}^2_{1.5} ) 
+
\ccPz
+ P(\norm{v}_{2.5+\de}) + \int_0^t \ccP.
\end{split}
\label{R_term_as_paper}
\end{align}
Integrating by parts the derivative $\partial_\be$, we have
\begin{align}
\begin{split}
\int_0^t \int_\Om   a_{\mu \ga} \partial^2_t \partial_\be v_\ga a_{\be \al}  \partial_\mu v_\al\partial^3_t q 
= &
- \int_0^t \int_\Om   a_{\mu \ga} \partial^2_t  v_\ga a_{\be \al}  \partial_\be\partial_\mu v_\al\partial^3_t q 
- \int_0^t \int_\Om   a_{\mu \ga} \partial^2_t  v_\ga a_{\be \al} \partial_\mu v_\al
 \partial_\be\partial^3_t q
 \\
 & 
 - \int_0^t \int_\Om    \partial_\be a_{\mu \ga} \partial^2_t  v_\ga a_{\be \al} \partial_\mu v_\al
\partial^3_t q
+ \int_0^t \int_{\Ga_1}   a_{\mu \ga} \partial^2_t  v_\ga a_{3 \al}  \partial_\mu v_\al\partial^3_t q,
\end{split}
\nonumber
\end{align}
where we used \eqref{div_identity} and  \eqref{Lagrangian_bry_v}.
The first three integrals on the right-hand side 
are estimated using a simple integration by parts in time
and are bounded by
\begin{align}
\begin{split}
& \widetilde{\epsilon}( \norm{\partial^2_t q}_1 + \norm{\partial^2_t v}^2_{1.5} )  
+\ccPz
+ P(\norm{v}_{2.5+\de})
+ \int_0^t \ccP,
\end{split}
\label{first_three_terms_as_in_paper}
\end{align}
whereas the last term is estimated as 
$I_{2121}$ (see \eqref{I_2121_def} and \eqref{I_2121_break_up}). 
Thus from \eqref{R_term_as_paper} and \eqref{first_three_terms_as_in_paper} we have
\begin{align}
\begin{split}
I_{213} \leq & \,
\widetilde{\epsilon}(  \norm{\partial^2_t q}_1 + \norm{\partial^2_t v}^2_{1.5} )  
+\ccPz
+ P(\norm{ v}_{2.5+\de})
+ \int_0^t \ccP.
\end{split}
\label{estimate_I_213}
\end{align}
Combining \eqref{I_21_break_up}, 
 \eqref{estimate_I_211}, \eqref{estimate_I_212}, and
\eqref{estimate_I_213}, we find
\begin{align}
\begin{split}
I_{21} \leq &
\, \widetilde{\epsilon} ( \norm{\partial^2_t q}_1^2 + \norm{ \partial^2_t v}^2_{1.5} )
+\ccPz
 +  P( \norm{v}_{2.5+\de}) 
+ \int_0^t  \ccP.
\end{split}
\label{estimate_I_21}
\end{align}

Next, we treat $I_{22}$, defined in \eqref{I_2_def}, which we write as
\begin{align}
\begin{split}
I_{22} = & \,
 3 \int_0^t \int_\Om \partial_t a_{\mu\al} \partial^2_t q \partial_\mu \partial^3_t v_\al 
 \\
=
& \,
3 \int_0^t \int_{\Ga_1} \partial_t a_{3\al} \partial^2_t q \partial^3_t v_\al 
-
 3 \int_0^t \int_\Om \partial_t a_{\mu\al} \partial_\mu \partial^2_t q \partial^3_t v_\al 
 \\
 =
 &
 -3 \int_0^t \int_{\Ga_1}  a_{3\ga} \partial_\nu v_\ga a_{\nu \al}  \partial^2_t q \partial^3_t v_\al 
-  3 \int_0^t \int_\Om \partial_t a_{\mu\al} \partial_\mu \partial^2_t q \partial^3_t v_\al 
 \\
 = &\,
 3\si  \int_0^t \int_{\Ga_1} \partial_\nu v_\ga a_{\nu \al}  
 \partial^2_t (\sqrt{g} \Delta_g \eta_\ga)  \partial^3_t v_\al 
 + 6 \int_0^t \int_{\Ga_1} \partial_\nu v_\ga a_{\nu \al}   \partial_t a_{3 \ga} \partial_t q  \partial^3_t v_\al 
\\&
+ 3  \int_0^t \int_{\Ga_1} \partial_\nu v_\ga a_{\nu \al}   \partial^2_t a_{3 \ga}  q  \partial^3_t v_\al 
 -  3 \int_0^t \int_\Om \partial_t a_{\mu\al} \partial_\mu \partial^2_t q  \partial^3_t v_\al 
 \\
 = &\,
 I_{221} + I_{222} + I_{223} + I_{224},
\end{split}
\nonumber
\end{align}
where in the second equality we integrated by parts in
$\partial_{\mu}$, in the third equality
we used \eqref{Lagrangian_free_Euler_a_eq}, and in the fourth equality
\eqref{Lagrangian_bry_q}. 
In the second inequality, we also applied
\eqref{conditions_Ga_0}.

To estimate $ I_{221}$ we use the identity \eqref{partial_Laplacian_eta_identity} 
As above, it suffices to consider the general
multiplicative structure of the integrands and thus we write the right-hand side 
of \eqref{partial_Laplacian_eta_identity} as
\begin{gather}
 \overline{\partial} (Q(\overline{\partial} \eta) \partial_t \overline{\partial} \eta ).
\label{partial_Laplacian_symbolic}
\end{gather}
Integrating by parts the $\overline{\partial}$ derivative in 
\eqref{partial_Laplacian_symbolic} and integrating by parts in time one of the derivatives
in the $\partial^3_t v$ term, we find
\begin{align}
\begin{split}
I_{221} = & 
\int_0^t \int_{\Ga_1} 
{\partial}v a 
\partial_t  \overline{\partial} (Q(\overline{\partial} \eta) \partial_t \overline{\partial} \eta )
\partial^3_t v
\\
= &
-\int_0^t \int_{\Ga_1} 
\partial_t  (Q(\overline{\partial} \eta) \partial_t \overline{\partial} \eta )
( \partial^3_t \overline{\partial} v {\partial} v a 
+ 
\partial^3_t v \overline{\partial} ( {\partial} v a ) )
\\
= &
-\left.
\int_{\Ga_1} 
\partial_t  (Q(\overline{\partial} \eta) \partial_t \overline{\partial} \eta )
\partial^2_t \overline{\partial} v {\partial} v a 
\right|_0^t
- \left.
\int_{\Ga_1} 
 \partial_t  (Q(\overline{\partial} \eta) \partial_t \overline{\partial} \eta )
\partial^2_t v \overline{\partial} ( {\partial} v a ) 
\right|_0^t
\\
&+
\int_0^t \int_{\Ga_1} 
 \partial^2_t  (Q(\overline{\partial} \eta) \partial_t \overline{\partial} \eta )
 \partial^2_t \overline{\partial} v a {\partial} v
 +
 \int_0^t \int_{\Ga_1} 
 \partial_t  (Q(\overline{\partial} \eta) \partial_t \overline{\partial} \eta )
 \partial^2_t \overline{\partial} v \partial_t (a {\partial} v)
 \\
 & 
 +
 \int_0^t \int_{\Ga_1} 
 \partial^2_t  (Q(\overline{\partial} \eta) \partial_t \overline{\partial} \eta )
 \partial^2_t  v  \overline{\partial}( a {\partial} v)
 +
 \int_0^t \int_{\Ga_1} 
 \partial_t  (Q(\overline{\partial} \eta) \partial_t \overline{\partial} \eta )
 \partial^2_t  v \partial_t \overline{\partial} (a {\partial} v)
   .
\end{split}
\nonumber
\end{align}
The estimate of these terms is done in a similar way as other boundary
integrals handled above. 
For example,
the first of the two pointwise terms above equals
\begin{align}
\begin{split}
&
-\left.
\int_{\Ga_1} 
Q(\overline{\partial} \eta) \partial^2_t \overline{\partial} \eta 
\partial^2_t \overline{\partial} v {\partial} v a 
\right|_0^t
-\left.
\int_{\Ga_1} 
Q(\overline{\partial} \eta) (\partial_t \overline{\partial} \eta )^2
\partial^2_t \overline{\partial} v {\partial} v a 
\right|_0^t
\\
&
\indeq
= 
-\left.
\int_{\Ga_1} 
Q(\overline{\partial} \eta) \partial_t \overline{\partial} v
\partial^2_t \overline{\partial} v {\partial} v a 
\right|_0^t
-\left.
\int_{\Ga_1} 
Q(\overline{\partial} \eta) (\overline{\partial} v )^2
\partial^2_t \overline{\partial} v {\partial} v a 
\right|_0^t
\\
&
\indeq\leq 
\, 
\widetilde{\epsilon} \norm{\partial^2_t v }^2_{1.5} 
+\ccPz
+ P(\norm{v}_{2.5+\de})
+
\int_0^t 
P ( \norm{v}_3, \norm{\partial^2_t v}_{1.5} ).
\end{split}
\nonumber
\end{align}
Other terms
are handled similarly; we do not present
the estimates here as they consist of a repetition of ideas used above. We obtain
\begin{align}
\begin{split}
I_{221} \leq & \,
\widetilde{\epsilon} \norm{\partial^2_t v }^2_{1.5} 
+ 
\ccPz
+
P( \norm{v}_{2.5+\de})
+
\int_0^t \ccP.
\end{split}
\nonumber
\end{align}
The terms $I_{222}$, $I_{223}$, and $I_{224}$ are also bounded by essentially
a repetition of the ideas presented so far and, thus, we again omit the details. 
We point out that a term in $\partial^3_t a$ appears in $I_{223}$ after integration
by parts in time. But this term is then handled as in $I_{21214}$, with the help
of  \eqref{a_3_ga_bry}. The final estimate for $I_{22}$ is
\begin{align}
\begin{split}
I_{22} \leq &
\, 
\widetilde{\epsilon} ( \norm{\partial^2_t v }^2_{1.5} + \norm{\partial^2_t q}^2_1 )
+ P(  \norm{v}_{2.5+\de})
+ \ccPz
+
\int_0^t \ccP
   .
\end{split}
\label{estimate_I_22}
\end{align}
Finally, the 
terms $I_{23}$ and $I_{24}$ 
are treated
similarly to the terms $I_{21}$
with the corresponding estimates similar to the ones that have already been performed
above. The final estimate
for  $I_{23}+I_{24}$ reads
\begin{align}
\begin{split}
I_{23} + I_{24} \leq &  \,
\widetilde{\epsilon} (\norm{\partial^2_t v}^2_{1.5}  + \norm{\partial^2_t q}^2_1   ) 
 + P(\norm{v}_{2.5+\de}) 
 +\ccPz
 + 
\int_0^t \ccP.
\end{split}
\label{estimate_I_23_24}
\end{align}
Collecting all the bounds above then concludes the proof of the
lemma.
\end{proof}

\section{Estimates at $t=0$\label{section_time_zero}}
The above estimates involve several quantities evaluated at time zero. In this section we show
that all such quantities may be estimated in terms of the initial data.


From \eqref{free_Euler_system} we have that $p$ satisfies
\begin{subequations}{}
\begin{alignat}{5}
\Delta p &&\, = \,& \, -\dive (\nabla_u u)  &&  \hspace{0.25cm}   \text{ in } && \ccD,
\nonumber
 \\
\frac{\partial p}{\partial 3} && \, = \, & \, 0
\nonumber
\\
p  && \, = \,& \, \si \cH  &&  \hspace{0.25cm}  \text{ on } && \, \Ga_1(t)
,
\nonumber
\end{alignat}
\end{subequations}
where $\Ga_1(t) = \eta(\Ga_1)$.
Evaluating at $t=0$ 
and denoting $q_0=q(0)$
gives
\begin{subequations}{}
\begin{alignat}{5}
\Delta q_0 &&\, = \,& \, -\partial_\al v_{0\be} \partial_\be v_{0\al}  &&  \hspace{0.25cm}   \text{ in } && \Om_0,
\nonumber
 \\
\partial_{3}p
 && \, 
= \, & \, 0 && \hspace{0.25cm}  \text{ on }&& \,  \Ga_0, 
\nonumber
\\
p  && \, = \,& \, 0 &&  \hspace{0.25cm}  \text{ on } && \, \Ga_1, 
\nonumber
\end{alignat}
\end{subequations}
where we used 
$\cH(0) = 0$ on $\Ga_1$. We thus obtain the estimate
\begin{gather}
\norm{ q_0 }_4 \leq C \norm{v_0}_3^2.
\nonumber
\end{gather}
Evaluating \eqref{Lagrangian_free_Euler_eq} at $t=0$ and using $a(0)=  I$
produces
\begin{gather}
\norm{ \partial_t v(0) }_3 \leq  C\norm{q_0}_4 \leq C \norm{v_0}_3^2.
\nonumber
\end{gather}
From  \eqref{Lagrangian_free_Euler_a_eq} and its time derivative we also obtain
\begin{gather}
\norm{ \partial_t a(0) }_2 \leq C \norm{ v_0 }_3
\nonumber
\end{gather}
and
\begin{gather}
\norm{ \partial^2_t a(0) }_2 \leq P( \norm{ v_0 }_3 ),
\nonumber
\end{gather}
after using the estimate for $\partial_t v(0)$.

From \eqref{press_estimate_eq} and \eqref{press_estimate_bry}, and
also using \eqref{div_identity} and \eqref{Lagrangian_bry_q},  we may write
\begin{subequations}{}
\begin{alignat}{5}
\partial_\mu( a_{\mu\nu} a_{\la \nu} \partial_\la q ) &&\, = \,& \,  \partial_t a_{\mu\la} \partial_\mu v_\la  &&  \hspace{0.25cm}   \text{ in } && \Om_0,
\nonumber
 \\
a_{3\nu} a_{\la \nu} \partial_\la q   
&& \, = \, & \, - a_{3 \nu} \partial_t v_\nu && \hspace{0.25cm}  \text{ on }&& \,  \Ga_0, 
\nonumber
\\
a_{3 \al} q   && \, 
= \,& \,  -\si \sqrt{g} \Delta_g \eta_\al  &&  \hspace{0.25cm}  \text{ on } && \, \Ga_1.
\nonumber
\end{alignat}
\end{subequations}
Setting $\al = 3$ in the boundary condition on $\Ga_1$, differentiating 
in time, and evaluating at $t=0$ yields
\begin{subequations}{}
\begin{alignat}{5}
\Delta \partial_t q(0) &&\, = \,& \,  
( -\partial_t a_{\al \mu} \partial^2_{\mu \al } q
- \partial_t \partial_\al a_{\be \al} \partial_\be q - \partial_t a_{\be\mu} \partial^2_{\mu \be} q 
 &&   && 
\nonumber
\\
&&\,  \,& \,  
+ \partial^2_t a_{\al\be} \partial_\al v_\be + \partial_t a_{\al\be} \partial_\al \partial_t
v_\be )\Big|_{t=0}
 &&  \hspace{0.25cm}   \text{ in } && \Om_0,
\nonumber
 \\
\partial_{3} \partial_t q(0)  && \, = \, & \, 
 - \partial_t a_{3\be} \partial_\be q \Big|_{t=0}
&& \hspace{0.25cm}  \text{ on }&& \,  \Ga_0, 
\nonumber
\\
\partial_t q(0) && \, = \,& \, 
- \partial_t a_{3 \al} q  \Big|_{t=0} 
- \si  \overline{\Delta} v_{03}
 &&  \hspace{0.25cm}  \text{ on } && \, \Ga_1, 
\nonumber
\end{alignat}
\end{subequations}
where we used $a(0) = I$, \eqref{div_identity}, and \eqref{conditions_Ga_0}.
Above, $\overline{\Delta}$ is the flat Laplacian on the boundary.
Using the estimates we have already obtained at $t=0$ and invoking the
elliptic
theory, we obtain
\begin{gather}
\norm{ \partial_t q(0) }_{2.5} \leq  P( \norm{ v_0}_3, \norm{ v_{03} }_{4,\Ga_1} )
\nonumber
\end{gather}
and thus
$\norm{ \partial_{tt} v(0) }_{1.5} \leq  P( \norm{ v_0}_3, \norm{
v_{03} }_{4,\Ga_1} )$
by differentiating
\eqref{Lagrangian_free_Euler_eq}
and evaluating at $t=0$.
Observe that if the initial data is irrotational
in a neighborhood of the 
interface $\Gamma_1$, 
we have $\bar\Delta v_{03}=0$ there, and
we simply get
$\norm{ \partial_t q(0) }_{2.5} \leq  P( \norm{ v_0}_3)$
and
$\norm{ \partial_t v(0) }_{1.5} \leq  P( \norm{ v_0}_3)$.

It is now clear how to proceed to bound higher time derivatives at the time zero.
We further differentiate the equations in time, evaluate at $t=0$, and use the 
estimates derived for lower time derivatives at time zero. We conclude that
\begin{align}
\begin{split}
 \norm{\partial_t v(0)}_{2.5} & + \norm{\partial^2_t v(0)}_{1.5}
+ \norm{\partial^3_t v(0)}_0
\\
&
+ \norm{q(0)}_3 + \norm{\partial_t q(0)}_{2.5} + \norm{\partial^2_t q(0)}_1 \leq 
 P( \norm{ v_0}_3, \norm{ v_{03}}_{4,\Ga_1} ).
\end{split}
\label{estimate_time_zero}
\end{align}
Combining 
\eqref{energy_identity_L_2}, \eqref{I_def},
 \eqref{I_2_def},
\eqref{estimate_I_1}, \eqref{estimate_I_21},
\eqref{estimate_I_22}, \eqref{estimate_I_23_24}, and
\eqref{estimate_time_zero}, we obtain
\eqref{partial_3_t_v_estimate}.

\begin{remark}
The arguments in this section clarify why 
our priori bounds also depend on a higher norm of 
$v_{03}$ 
on the boundary.
We see from the above boundary value problem for $\partial_t q(0)$ that this term is 
controlled by  $\overline{\Delta} v_3 = \overline{\Delta} (v \cdot N)$ on the boundary
$\Ga_1$. Since we want to bound $\partial^2_t v$ in $H^{1.5}$, we need $\partial^2_t v$
in $H^1(\Ga_1)$. By \eqref{Lagrangian_free_Euler_eq} differentiated in time, this requires,
considering the tangent component of this equation, that $\partial_t q(0)$ 
belongs to $H^2(\Ga_1)$.
Thus, one needs $\overline{\Delta} 
v_{3}
\in H^2(\Ga_1)$.
\label{remark_extra_regularity}
\end{remark}

\section{Energy estimate on the two times differentiated system}\label{section_H_1_estimate}
In this section we prove the next lemma, which
provides an upper bound on the $\dot H^{2}(\Gamma_1)$ norm
of $\Pi\partial_{t} v$.
In the final section, we show that this leads to a bound 
on $\Vert \partial_{t}v\Vert_{H^{2}(\Gamma_1)}$.

\begin{lemma}
\label{L02}
We have
\begin{align}
\begin{split}
\norm{\overline{\partial} \partial^2_t v}_0^2 +  
\norm{ \overline{\partial}{}^2 (\Pi \partial_t v)}^2_{0,\Ga_1}
\leq & 
\,
\widetilde{\epsilon}  \norm{ \partial_t v}_{2.5}^2
+\ccPz
+ P( \norm{v}_3,\norm{q}_{1.5,\Ga_1})
+
\int_0^t \ccP.
\end{split}
\label{partial_2_t_v_estimate}
\end{align}
\end{lemma}

\begin{proof}[Proof of Lemma~\ref{L02}]
We apply $\partial_m \partial^2_t$ to \eqref{Lagrangian_free_Euler_eq}, contract with
$\partial_m \partial^2_t v_\al$, integrate in time and space and find
\begin{align}
\begin{split}
\frac{1}{2} \int_\Om \partial_m \partial^2_t v_\al \partial_m \partial^2_t v_\al 
= & 
\int_0^t \int_\Om  \partial_\mu \partial_m \partial^2_t v_\al a_{\mu\al} \partial^2_t \partial_m q 
- \int_0^t \int_{\Ga_1} \partial_m \partial^2_t v_\al a_{3 \al} \partial^2_t \partial_m q
\\
&
 - \int_0^t \int_\Om \partial_m \partial^2_t v_\al r_{\al m}
 +\left. \frac{1}{2} \int_\Om \partial_m \partial^2_t v_\al \partial_m \partial^2_t v_\al 
\right|_{t=0} 
\\
=& J_1 + J_2 + J_3 + J_0
,
\end{split}
\nonumber
\end{align}
where in the second equality we integrated by parts in $\partial_{m}$
and used \eqref{div_identity} with \eqref{conditions_Ga_0}; 
we also denoted
\begin{align}
\begin{split}
r_{\al m} = & \partial^2_t \partial_m a_{\mu\al}\partial_\mu q
+ 2 \partial_t \partial_m a_{\mu \al} \partial_t \partial_\mu q
+ \partial_m a_{\mu\al} \partial^2_t \partial_\mu q + \partial^2_t a_{\mu \al} \partial_\mu 
\partial_m q
+ 2 \partial_t a_{\mu \al}\partial_t \partial_\mu \partial_m q.
\end{split}
\nonumber
\end{align}
The treatment of the above terms is parallel to the previous proof 
(the derivatives $\partial_{t}^{3}$ are replaced by
$\partial_{t}^2\partial_{m}$), and thus we only sketch the details.
Since \eqref{Lagrangian_free_Euler_div} implies $\partial_m
\partial^2_t (a_{\mu\al} \partial_\mu v_\al) = 0$, we have
\begin{align}
\begin{split}
J_ 1=& 
- \int_0^t \int_\Om \partial_m a_{\mu \al} \partial_\mu \partial^2_t v_\al 
\partial^2_t \partial_m q
- 2 \int_0^t \int_\Om \partial_m \partial_t  a_{\mu \al} \partial_\mu \partial_t v_\al 
\partial^2_t \partial_m q
\\
&
- 2 \int_0^t \int_\Om \partial_t a_{\mu \al}  \partial_m  \partial_\mu \partial_t v_\al 
\partial^2_t \partial_m q
- \int_0^t \int_\Om \partial_m \partial^2_t a_{\mu \al} \partial_\mu v_\al 
\partial^2_t \partial_m q
\\
&
- \int_0^t \int_\Om  \partial^2_t a_{\mu \al} \partial_m \partial_\mu v_\al 
\partial^2_t \partial_m q
   .
\end{split}
\nonumber
\end{align}
All the integrals are bounded 
by $\int_{0}^{t}\ccP$
using 
H\"older and Sobolev inequalities.
Similarly, we have
$J_0\leq \ccPz$.
Also, the term $J_3$ is treated by the same procedure just used for $J_1$,
yielding
$J_3\leq \int_{0}^{t}\ccP$.

The remaining
term 
$J_2
= 
-\int_{0}^{t}\int_{\Gamma_1}\partial_m\partial_{t}^2 v_{\alpha} a_{3\alpha}\partial_{t}^2\partial_{k}q
$
may be
rewritten as
\begin{align}
\begin{split}
J_2 = & 
- \int_0^t \int_{\Ga_1} \partial_m \partial^2_t (a_{3 \al} q) \partial_m \partial^2_t v_\al 
+\int_0^t \int_{\Ga_1} \partial_m (\partial^2_t a_{3\al} q ) \partial_m \partial^2_t v_\al 
\\
& 
+ 2\int_0^t \int_{\Ga_1} \partial_m (\partial_t a_{3\al} \partial_t q  )\partial_m \partial^2_t v_\al 
 + \int_0^t \int_{\Ga_1} \partial_m a_{3\al} \partial^2_t q \partial_m \partial^2_t v_\al 
 \\
 = &
 J_{21} + J_{22} + J_{23} + J_{24}.
\end{split}
\label{J_2_break_up}
\end{align}
To estimate $J_{21}$, we use \eqref{Lagrangian_bry_q} and obtain
\begin{gather}
J_{21} = \si\int_0^t \int_{\Ga_1} 
\partial_m \partial^2_t (\sqrt{g} \Delta_g \eta_\al ) \partial_m \partial^2_t v_\al .
\nonumber
\end{gather}
This term is now bounded in the same fashion as $I_{1}$ in 
the previous proof.
Namely, with $\partial_m \partial^2_t$ replacing $\partial^3_t$,
we use \eqref{partial_Laplacian_eta_identity} and split $J_{21}$ similarly to
\eqref{I_1_break_up}, yielding
\begin{gather}
J_{21} = J_{211} + J_{212} + J_{213} + J_{214},
\label{J_21_break_up}
\end{gather}
where $J_{211}$ is the analogue of $I_{11}$, $J_{212}$ the analogue of $I_{12}$,
$J_{213}$ the analogue of $I_{13}+I_{14}$, and $J_{214}$ the analogue of 
$I_{15}+I_{16}$. The terms $I_{211}$ and $I_{212}$ contain the products of
top derivatives and are estimated in the same way as $I_{11}$ and $I_{12}$.
The remainders $J_{213}$ and $J_{214}$ need to be treated in a slightly different
fashion than $I_{13}+I_{14}$ and $I_{15}+I_{16}$ because of the different ways the 
time and space derivatives in $\partial_t\overline{\partial}$ may be distributed 
in \eqref{partial_Laplacian_eta_identity} after taking $\overline{\partial}_A = \partial_t$.
We have
\begin{align}
\begin{split}
J_{213} = &
 -\int_0^t \int_{\Ga_1} \overline{\partial} Q(\overline{\partial} \eta)
\partial_t \overline{\partial} v \overline{\partial} \partial^2_t \overline{\partial} v
 -\int_0^t \int_{\Ga_1} \partial_t Q(\overline{\partial} \eta)
\overline{\partial}{}^2 v \overline{\partial} \partial^2_t \overline{\partial} v
=
J_{2131} + J_{2132}
\end{split}
\nonumber
\end{align}
and
\begin{align}
\begin{split}
J_{214} 
= & -\int_0^t \int_{\Ga_1} \overline{\partial} \partial_t Q(\overline{\partial} \eta)
\overline{\partial} v \overline{\partial} \partial^2_t \overline{\partial} v,
\end{split}
\nonumber
\end{align}
where we followed the symbolic notation in Remark~\ref{remark_rational}.
Using \eqref{partial_Q_symb} and integrating by parts in space
\begin{align}
\begin{split}
J_{2131} = & 
 \int_0^t \int_{\Ga_1} Q(\overline{\partial} \eta)
  \overline{\partial}{}^2 \eta 
\partial_t \overline{\partial}{}^2 v \partial^2_t \overline{\partial} v
 +\int_0^t \int_{\Ga_1} Q(\overline{\partial} \eta)
  \overline{\partial}{}^3 \eta 
\partial_t \overline{\partial} v \partial^2_t \overline{\partial} v
\\
&
 +\int_0^t \int_{\Ga_1} Q(\overline{\partial} \eta)
  (\overline{\partial}{}^2 \eta)^2 
\partial_t \overline{\partial} v \partial^2_t \overline{\partial} v
\\
\leq  &
\int_0^t \norm{ Q(\overline{\partial} \eta) \overline{\partial}{}^2 \eta}_{L^\infty(\Ga_1)}
\norm{\partial_t \overline{\partial}{}^2 v}_{0,\Ga_1} \norm{\partial^2_t \overline{\partial} v }_{0,\Ga_1}
\\
&
+ \int_0^t \norm{ Q(\overline{\partial} \eta) }_{L^\infty(\Ga_1)}
\norm{\overline{\partial}{}^3 \eta}_{L^4(\Ga_1)} 
\norm{\partial_t \overline{\partial} v}_{L^4(\Ga_1)} \norm{\partial^2_t \overline{\partial} v }_{0,\Ga_1}
\\
& 
+
\int_0^t \norm{\overline{\partial}{}^2 \eta}_{L^\infty(\Ga_1)}^2 \norm{\partial_t \overline{\partial}
v}_{0,\Ga_1} \norm{\partial^2_t \overline{\partial} v}_{0,\Ga_1}
\leq 
\int_{0}^{t}\ccP
\end{split}
\nonumber
\end{align}
where we used \eqref{rational_estimate}, \eqref{rational_estimate_2}, and Proposition
\ref{proposition_regularity}.
Next, integrating by parts in time
\begin{align}
\begin{split}
J_{2132} = &
-\int_0^t \int_{\Ga_1} Q(\overline{\partial} \eta) \overline{\partial} v
\overline{\partial}{}^2 v \overline{\partial} \partial^2_t \overline{\partial} v
\\
= &
- \int_{\Ga_1}  Q(\overline{\partial} \eta) \overline{\partial} v
\overline{\partial}{}^2 v  \partial_t \overline{\partial}{}^2 v
+ J_{2132,0} 
+ \int_0^t \int_{\Ga_1} \partial_t ( Q(\overline{\partial} \eta) \overline{\partial} v
\overline{\partial}{}^2 v  ) \partial_t \overline{\partial}{}^2 v
\\
= &
- \int_{\Ga_1}  Q(\overline{\partial} \eta) \overline{\partial} v
\overline{\partial}{}^2 v  \partial_t \overline{\partial}{}^2 v
+ J_{2132,0} 
+ \int_0^t \int_{\Ga_1} Q(\overline{\partial} \eta) (\overline{\partial} v)^2
\overline{\partial}{}^2 v   \partial_t \overline{\partial}{}^2 v
\\
&  
+ \int_0^t \int_{\Ga_1} Q(\overline{\partial} \eta) \partial_t \overline{\partial} v
\overline{\partial}{}^2 v   \partial_t \overline{\partial}{}^2 v
+ \int_0^t \int_{\Ga_1} Q(\overline{\partial} \eta) \overline{\partial} v
\partial_t \overline{\partial}{}^2 v  \partial_t \overline{\partial}{}^2 v,
\end{split}
\nonumber
\end{align}
where
\begin{gather}
 J_{2132,0} = 
 \left. \int_{\Ga_1}  Q(\overline{\partial} \eta) \overline{\partial} v
\overline{\partial}{}^2 v  \partial_t \overline{\partial}{}^2 v 
\right|_{t=0}
\leq \ccPz.
\nonumber
\end{gather}
A direct estimate of each term in $J_{2132}$ now produces
\begin{align}
\begin{split}
J_{213} \leq & 
\, \widetilde{\epsilon} \norm{\partial_t v}_{2.5}^2 + P(\norm{v}_{2.5+\de})
+\ccPz
+ \int_0^t  \ccP.
\end{split}
\label{estimate_J_213}
\end{align}
Moving to $J_{214}$ and recalling \eqref{partial_Q_symb},
\begin{align}
\begin{split}
J_{214} = 
&
-\int_0^t \int_{\Ga_1} \overline{\partial} \partial_t Q(\overline{\partial} \eta)
\overline{\partial} v \overline{\partial} \partial^2_t \overline{\partial} v
\\ 
=&
-\int_0^t \int_{\Ga_1}  Q(\overline{\partial} \eta)
\overline{\partial}{}^2 \eta (\overline{\partial} v)^2 
 \overline{\partial} \partial^2_t \overline{\partial} v
-\int_0^t \int_{\Ga_1}  Q(\overline{\partial} \eta)
\overline{\partial}{}^2 v
\overline{\partial} v \overline{\partial} \partial^2_t \overline{\partial} v
=
J_{2141} + J_{2142}.
\end{split}
\nonumber
\end{align}
Integrating by parts in space, we find
\begin{align}
\begin{split}
J_{2141} = & 
\int_0^t \int_{\Ga_1}  Q(\overline{\partial} \eta)
(\overline{\partial}{}^2 \eta)^2 (\overline{\partial} v)^2 
 \partial^2_t \overline{\partial} v
 +
 \int_0^t \int_{\Ga_1}  Q(\overline{\partial} \eta)
\overline{\partial}{}^3 \eta (\overline{\partial} v)^2 
 \partial^2_t \overline{\partial} v
 \\
 & +
  \int_0^t \int_{\Ga_1}  Q(\overline{\partial} \eta)
\overline{\partial}{}^2 \eta \overline{\partial} v 
\overline{\partial}{}^2 v
 \partial^2_t \overline{\partial} v.
\end{split}
\nonumber
\end{align}
With the help of  \eqref{rational_estimate}, \eqref{rational_estimate_2}, and Proposition
\ref{proposition_regularity}, we find
\begin{align}
\begin{split}
J_{2141} \leq & \,
C\int_0^t 
\norm{Q(\overline{\partial} \eta)}_{L^\infty(\Ga_1)} 
\norm{ (\overline{\partial}{}^2 \eta)^2}_{0,\Ga_1}
\norm{(\overline{\partial} v)^2}_{L^\infty(\Ga_1)} 
\norm{\partial^2_t \overline{\partial} v}_{0,\Ga_1}
\\
&
+
C\int_0^t 
\norm{Q(\overline{\partial} \eta)}_{L^\infty(\Ga_1)} 
\norm{ \overline{\partial}{}^3 \eta}_{0,\Ga_1}
\norm{(\overline{\partial} v)^2}_{L^\infty(\Ga_1)} 
\norm{\partial^2_t \overline{\partial} v}_{0,\Ga_1}
\\
&
+
C\int_0^t 
\norm{Q(\overline{\partial} \eta)}_{L^\infty(\Ga_1)} 
\norm{ \overline{\partial}{}^2 \eta}_{L^4(\Ga_1)}
\norm{\overline{\partial} v}_{L^\infty(\Ga_1)} 
\norm{\overline{\partial}{}^2 v}_{L^4(\Ga_1)} 
\norm{\partial^2_t \overline{\partial} v}_{0,\Ga_1}
\leq
\int_{0}^{t}\ccP,
\end{split}
\nonumber
\end{align}
where in the first integral on the right-hand side we used
\begin{gather}
\norm{ (\overline{\partial}{}^2 \eta)^2}_{0,\Ga_1} \leq
C\norm{ (\overline{\partial}{}^2 \eta)^2}_{1.5,\Ga_1}
\leq C\norm{ \eta}^2_{3.5,\Ga_1}.
\nonumber
\end{gather}
Next, integrating by parts in time and estimating as above, we find
\begin{align}
\begin{split}
 J_{2142} = & 
 - \int_{\Ga_1}  Q(\overline{\partial} \eta)
\overline{\partial}{}^2 v
\overline{\partial} v  \partial_t \overline{\partial}{}^2 v + 
\left.
 \int_{\Ga_1}  Q(\overline{\partial} \eta)
\overline{\partial}{}^2 v
\overline{\partial} v  \partial_t \overline{\partial}{}^2 v \right|_{t=0}
\\
&
+
\int_0^t \int_{\Ga_1}  Q(\overline{\partial} \eta)
\overline{\partial} v
\overline{\partial}{}^2 v
\overline{\partial} v  \partial_t \overline{\partial}{}^2 v
+ 
\int_0^t \int_{\Ga_1}  Q(\overline{\partial} \eta)
\partial_t \overline{\partial}{}^2 v
\overline{\partial} v  \partial_t \overline{\partial}{}^2 v
\\
&
+ 
\int_0^t \int_{\Ga_1}  Q(\overline{\partial} \eta)
\overline{\partial}{}^2 v
\partial_t \overline{\partial} v  \partial_t \overline{\partial}{}^2 v
\\
\leq &
\, 
\widetilde{\epsilon} \norm{\partial_t v}^2_{2.5} + P(\norm{v}_{2.5+\de})
+
\ccPz
+ \int_0^t \ccP.
\end{split}
\nonumber
\end{align}
Therefore,
\begin{align}
\begin{split}
J_{214} \leq &
\, \widetilde{\epsilon} \norm{\partial_t v}^2_{2.5} + P( \norm{v}_{2.5+\de}) 
+
\ccPz
+\int_{0}^{t}\ccP.
\end{split}
\label{estimate_J_214}
\end{align}
Combining \eqref{J_21_break_up}, \eqref{estimate_J_213},
and~\eqref{estimate_J_214}, we find
\begin{align}
\begin{split}
J_{21}  \leq &
- \frac{1}{C_{21}} \norm{ \overline{\partial}{}^2 (\Pi \partial_t v)}^2_{0,\Ga_1}
+
\widetilde{\epsilon}  \norm{ \partial_t v}_{2.5}^2
+\ccPz
+ P(  \norm{v}_{2.5+\de},\norm{q}_{1.5,\Ga_1})
+
\int_{0}^{t}\ccP.
\end{split}
\label{estimate_J_21}
\end{align}


We now sketch the estimates of the remaining boundary terms 
$J_{22}$, $J_{23}$, and $J_{24}$
in \eqref{J_2_break_up}.
The key is to use \eqref{a_3_ga_bry} and Proposition~\ref{proposition_regularity}.
Recalling that $N=(0,0,1)$, we have
\begin{align}
\begin{split}
J_{22} = &
\int_0^t \int_{\Ga_1} \partial_m \partial^2_t a_{3\al} q \partial_m \partial^2_t v_\al 
+\int_0^t \int_{\Ga_1} \partial^2_t a_{3\al} \partial_m  q  \partial_m \partial^2_t v_\al 
\\
\leq &
C \sum_{\al=1}^3 \int_0^t \norm{q}_{L^\infty(\Ga_1)} \norm{\overline{\partial}\partial^2_t  a_{3\al} }_{0,\Ga_1}
\norm{\overline{\partial} \partial^2_t v}_{0,\Ga_1}
+ 
C \sum_{\al=1}^3 \int_0^t \norm{\overline{\partial} q}_{L^\infty(\Ga_1)} \norm{
\partial^2_t a_{3\al} }_{0,\Ga_1}
\norm{ \overline{\partial} \partial^2_t v}_{0,\Ga_1}.
\end{split}
\nonumber
\end{align}
From \eqref{a_3_ga_bry},  we have, written symbolically,
\begin{gather}
\partial^2_t a_{3\al} = (\overline{\partial} v)^2
 + \overline{\partial} \eta \overline{\partial} \partial_t v
 \, \text{ on } \Ga_1
   .
\nonumber
\end{gather}
Differentiating this identity, we get
easily
$
J_{22} 
\leq 
\int_{0}^{t}\ccP
$,
where we used Proposition~\ref{proposition_regularity}.
The terms $J_{23}$ and $J_{24}$ are handled similarly
leading to
\begin{align}
\begin{split}
J_{22} + J_{23} + J_{24} 
\leq 
\int_{0}^{t}\ccP
.
\end{split}
\label{estimate_J_22_23_24}
\end{align}
Using \eqref{J_2_break_up}, \eqref{estimate_J_21}, 
and \eqref{estimate_J_22_23_24}, we find
\begin{align}
\begin{split}
J_{2}  \leq &
- \frac{1}{C_{21}} \norm{ \overline{\partial}{}^2 (\Pi \partial_t v)}^2_{0,\Ga_1}
+
\widetilde{\epsilon}  \norm{ \partial_t v}_{2.5}^2
+ P( \norm{v}_{2.5+\de},\norm{q}_{1.5,\Ga_1})
+
\int_0^t \ccP.
\end{split}
\label{estimate_J_2}
\end{align}
Finally, combining 
all the inequalities concludes the proof.
\end{proof}

\section{Boundary estimate on $v_3$\label{section_H_2.5_estimate}}
The purpose of this section is to establish control of
$v_{3}$ on the free boundary $\Gamma_1$.

\begin{lemma}
\label{L03}
The restriction of the third component of the velocity 
to $\Gamma_1$
satisfies
\begin{align}
\begin{split}
\norm{ v_ 3}_{2.5,\Ga_1} 
\leq &
C \norm{\partial_t q}_1 + P(\norm{v}_{2.5+\de}, \norm{q}_{1.5} ) 
+ \int_{0}^{t} \ccP.
\end{split}
\label{estimate_v_2.5_boundary}
\end{align} 
\end{lemma}

\begin{remark}
The estimates \eqref{partial_3_t_v_estimate} and  \eqref{partial_2_t_v_estimate} 
provide bounds for the normal components of $\partial^2_t v$ and $\partial_t v$. These enter in the div-curl estimates
below, thus providing bounds for  $\partial^2_t v$ and $\partial_t v$
in the interior. In order to bound $\norm{v}_3$ with the div-curl
estimates, we need to estimate $\norm{v}_{2.5,\Ga_1}$ first. For this we use \eqref{Lagrangian_bry_q} and Proposition~\ref{proposition_regularity}
rather than an energy estimate as in 
Sections~\ref{section_L_2_estimate}
and~\ref{section_H_1_estimate}. This is because
arguing as in those sections produces 
$\overline{\partial}{}^3 v$ on $\Ga_1$, for which we do not have 
a good control.
\end{remark}

{\begin{proof}[Proof of Lemma~\ref{L03}]
Differentiating \eqref{Lagrangian_bry_q} in time  and setting $\al=3$ yields
\begin{align}
\begin{split}
\sqrt{g} g^{ij} \partial^2_{ij} v_3 - \sqrt{g} g^{ij} \Ga_{ij}^k \partial_k v_3 
= & - \partial_t(\sqrt{g} g^{ij} ) \partial^2_{ij} \eta_3 - \partial_t( \sqrt{g} g^{ij} \Ga_{ij}^k ) \partial_k \eta_3 
- \frac{1}{\si} \partial_t a_{3 3} q - \frac{1}{\si} a_{3 \si} \partial_t q \, \text{ on } \, \Ga_1
 ,
\end{split}
\nonumber
\end{align}
where we also used
\eqref{Laplacian_identity_standard}.
In light of Proposition~\ref{proposition_regularity}, we have
\begin{gather}
\norm{ g_{ij} }_{2.5,\Ga_1},
\norm{ \Ga_{ij}^k }_{1.5,\Ga_1}
\leq C.
\end{gather}
Thus, by the elliptic estimates for operators with coefficients bounded in Sobolev norms (see \cite{CoutandShkollerFreeBoundary,EbenfeldEllipticRegularity})
we get
\begin{align}
\begin{split}
\norm{ v_3 }_{2.5,\Ga_1} \leq 
&
\, C
\norm{ \partial_t(\sqrt{g} g^{ij} ) \partial^2_{ij} \eta_3 }_{0.5,\Ga_1}
+ C \norm{ \partial_t( \sqrt{g} g^{ij} \Ga_{ij}^k ) \partial_k \eta_3 }_{0.5,\Ga_1}
\\
&
+C\norm{\partial_t a_{3 3} q}_{0.5,\Ga_1}
+C\norm{ a_{3 \si} \partial_t q }_{0.5,\Ga_1}
   .
\end{split}
\nonumber
\end{align}
For the first term on the right side,
\begin{align}
\begin{split}
\norm{ \partial_t(\sqrt{g} g^{ij} ) \partial^2_{ij} \eta_3 }_{0.5,\Ga_1} \leq &
C \norm{ \partial_t (\sqrt{g} g^{-1} )}_{1.5,\Ga_1} \norm{\overline{\partial}{}^2 \eta_3}_{0.5,\Ga_1} 
\leq
C \norm{ Q(\overline{\partial} \eta ) \partial_t \overline{\partial} \eta}_{1.5,\Ga_1} \norm{\overline{\partial}{}^2 \eta_3}_{0.5,\Ga_1} 
\\\leq & 
C\norm{v}_3 \norm{\overline{\partial}{}^2 \eta_3}_{0.5,\Ga_1} 
\leq
C\norm{v}_3 ,
\end{split}
\nonumber
\end{align}
where we used \eqref{rational_estimate}
and Lemma~\ref{lemma_auxiliary_long}(i).
For the next term, we use  \eqref{metric_def}, \eqref{Christoffel_identity}, and  \eqref{identity_det_g} to compute
\begin{align}
\begin{split}
\partial_t (\sqrt{g} g^{ij}  \Ga_{ij}^k ) = &
-\frac{1}{\sqrt{g}} g^{pq} \partial_p \eta_\tau \partial_t \partial_q \eta_\tau g^{ij} g^{kl} \partial_l \eta_\nu \partial^2_{ij} \eta_\nu
+ \sqrt{g} \partial_t( g^{ij} g^{kl} \partial_l \eta_\nu ) \partial^2_{ij} \eta_\nu 
+
\sqrt{g} g^{ij} g^{kl} \partial_l \eta_\nu \partial^2_{ij} \partial_t \eta_\nu.
\end{split}
\nonumber
\end{align}
Then
\begin{align}
\begin{split}
\nnorm{ \frac{1}{\sqrt{g}} g^{pq} \partial_p \eta_\tau \partial_t \partial_q \eta_\tau g^{ij} g^{kl} \partial_l \eta_\nu \partial^2_{ij} \eta_\nu}_{0.5,\Ga_1}
\leq 
&
\nnorm{ \frac{1}{\sqrt{g}} g^{pq} \partial_p \eta_\tau \partial_t \partial_q \eta_\tau g^{ij} g^{kl} \partial_l \eta_\nu}_{1.5,\Ga_1} 
\norm{ \partial^2_{ij} \eta_\nu}_{0.5,\Ga_1}
\\
\leq & P(\norm{\overline{\partial} \eta}_{1.5,\Ga_1} ) \norm{\overline{\partial} v}_{1.5,\Ga_1} \norm{\overline{\partial}{}^2 \eta}_{0.5,\Ga_1}
\leq 
C \norm{v}_3.
\end{split}
\nonumber
\end{align}
Also,
\begin{align}
\begin{split}
\norm{\sqrt{g} \partial_t( g^{ij} g^{kl} \partial_l \eta_\nu ) \partial^2_{ij} \eta_\nu }_{0.5,\Ga_1} 
\leq &
\norm{ P(\overline{\partial} \eta ) \partial_t \overline{\partial} \eta }_{1.5,\Ga_1} 
\norm{\overline{\partial}{}^2 \eta}_{0.5,\Ga_1}
\leq C \norm{v}_3
\end{split}
\nonumber
\end{align}
and
\begin{align}
\begin{split}
\norm{\sqrt{g} g^{ij} g^{kl} \partial_l \eta_\nu \partial^2_{ij} \partial_t \eta_\nu}_{0.5,\Ga_1} 
\leq & P(\norm{\overline{\partial} \eta }_{1.5,\Ga_1} ) \norm{ \overline{\partial}{}^2 v}_{0.5,\Ga_12}
\leq  C \norm{v}_3,
\end{split}
\nonumber
\end{align}
where we used  \eqref{identity_g_inverse} to compute $ \partial_t( g^{ij} g^{kl} \partial_l \eta_\nu ) $,
and  \eqref{rational_estimate} and \eqref{rational_estimate_2} have  also been employed.
Therefore,
\begin{gather}
\norm{ \partial_t (\sqrt{g} g^{ij}  \Ga_{ij}^k ) \partial_k \eta_3 }_{0.5,\Ga_1}
\leq C \norm{ \partial_t (\sqrt{g} g^{ij}  \Ga_{ij}^k )  }_{0.5,\Ga_1} \norm{ \partial_k \eta_3 }_{1.5,\Ga_1}
\leq C \sum_{k=1}^2 \norm{v}_3 \norm{ \partial_k \eta_3 }_{1.5,\Ga_1}
\leq C \norm{v}_3
.
\nonumber
\end{gather}
The terms containing $\partial_t a$ and $\partial_t q$ are easily
estimated, and 
\eqref{estimate_v_2.5_boundary} is proven.
\end{proof}

\section{Div-curl estimates\label{section_div_curl}}
In this section we derive estimates for $\norm{v}_3$, 
$\norm{\partial_t v}_{2.5}$, and $\norm{\partial^2_t v}_{1.5}$,
which we summarize in the next statement.

\begin{lemma}
\label{L04}
The velocity $v$  satisfies
\begin{align}
\begin{split}
\norm{ v}_3 \leq C( \norm{v_0}_0 
+ \norm{ \curl v_0}_2 
+ \norm{v_3}_{2.5,\Ga_1})
+ \ccP\int_0^t \ccP  
,
\end{split}
\label{estimate_v_div_curl}
\end{align}
for its time derivative we have
\begin{align}
\begin{split}
\norm{ \partial_t v}_{2.5} 
\leq &  
\ccPz
+ \norm{\partial_t a}_2 \norm{v}_{2.5} + \norm{ \partial_t v_3 }_{2,\Ga_1}
+\ccP\int_{0}^{t}\ccP,
\end{split}
\label{estimate_partial_t_v_div_curl}
\end{align}
while its second time derivative satisfies
\begin{align}
\begin{split}
\norm{ \partial^2_t v}_{1.5} 
\leq  & 
\ccPz
+ \norm{\partial_t a}_2 \norm{\partial_t v}_{1.5} 
+  \norm{\partial^2_t a}_{0.5} \norm{ v}_{2.5+\de} 
+ \norm{ \partial^2_t v_3 }_{1,\Ga_1}
+ \norm{\partial_t v}_{1.5} \norm{v}_{2.5+\de} 
+ \ccP\int_{0}^{t}\ccP
.
\end{split}
\label{estimate_partial_2_t_v_div_curl}
\end{align}
\end{lemma}

The arguments below are similar to those in 
\cite{IgorMihaelaSurfaceTension,KukavicaTuffahaVicol-3dFreeEuler}.

{\begin{proof}[Proof of Lemma~\ref{L04}]
First, we recall the inequality
\begin{align}
\begin{split}
\norm{ X }_s \leq C ( \norm{\curl X}_{s-1} + \norm{\dive X }_{s-1}
+ \norm{ X \cdot N }_{s - 0.5,\partial} + \norm{X}_0 ),
\end{split}
\label{div_curl_estimate}
\end{align}
for $s \geq 1$ and any vector field $X$ on $\Om$ for which the right-hand
side is well-defined (cf.~\cite{ShkollerElliptic,CoutandShkollerFreeBoundary}), 
and the Cauchy invariance 
(see e.g.~\cite{BesseFrischCauchyInvariant,KukavicaTuffahaVicol-3dFreeEuler})
\begin{gather}
\epsilon_{\al \be \ga} \partial_\be v_\mu \partial_\ga \eta_\mu 
= (\curl v_0)_\al,
\label{Cauchy_invariance}
\end{gather}
where $\epsilon_{ \al \be \ga}$ is the Levi-Civita symbol.
First, from \eqref{Lagrangian_free_Euler_div} we have
\begin{align}
\begin{split}
\norm{ \dive v }_2 = & \norm{ (\de_{\al\be} - a_{\al\be} ) \partial_\al v_\be }_2
\leq 
C \norm{ \de_{\al\be} - a_{\al\be} }_2  \norm{ \partial_\al v_\be }_2 
\leq
\, 
\widetilde{\epsilon} \norm{v}_3.
\end{split}
\nonumber
\end{align}
Next, using \eqref{Cauchy_invariance},
\begin{align}
\begin{split}
(\curl v)_\al = & \epsilon_{\al\be\ga} \partial_\be v_\ga 
=
\epsilon_{\al\be\ga} \partial_\be v_\ga  - \epsilon_{\al\be\ga} \partial_\be v_\mu \partial_\ga 
\eta_\mu + (\curl v_0)_\al 
\\
= &
\epsilon_{\al \be \ga} \partial_\be v_\mu (\de_{\mu\ga} - \partial_\ga \eta_\mu) 
+ (\curl v_0)_\al .
\end{split}
\nonumber
\end{align}
But 
\begin{gather}
\de_{\mu\ga} - \partial_\ga\eta_\mu = - \int_0^t \partial_\ga \partial_t \eta_\mu 
= - \int_0^t \partial_\ga v_\mu,
\nonumber
\end{gather}
so that 
\begin{gather}
\norm{\curl v}_2 
\leq 
\norm{ \curl v_0}_2
+
C \norm{v}_3 \int_0^t \norm{v}_3 
\leq 
\norm{ \curl v_0}_2
+
\ccP \int_0^t \ccP
.
\nonumber
\end{gather}
Since
\begin{gather}
\norm{ v}_0 \leq \norm{v_0}_0 + C \int_0^t \norm{ \partial_t v}_0
            \leq \norm{v_0}_0 +  \int_0^t \ccP,
\nonumber
\end{gather}
we get 
\eqref{estimate_v_div_curl} by
invoking  \eqref{div_curl_estimate}.
The inequalities 
\eqref{estimate_partial_t_v_div_curl}
and
\eqref{estimate_partial_2_t_v_div_curl}
are obtained analogously.
\end{proof}

\section{Closing the estimates\label{section_closing}}
 
In this section we finally collect all our estimates together to obtain
Theorem~\ref{main_theorem}. 

\subsection{Comparing $\Pi X$ and $X_3$}

From \eqref{estimate_partial_t_v_div_curl} and \eqref{estimate_partial_2_t_v_div_curl},
we see that we need to connect 
$\norm{ X_3 }_{s - 0.5,\partial}$ 
with the norm
of $\Pi X$ which entered in the energy estimates;
cf.~the second term on the left side of~\eqref{partial_3_t_v_estimate}.
The two necessary inequalities are stated next.

\begin{lemma}
\label{L05}
We have
\begin{align}
\begin{split}
\norm{ \overline{\partial} X_3 }_{0,\Ga_1}
\leq \widetilde{\epsilon} \left(
\sum_{\al=1}^3 \norm{ \overline{\partial} X_\al }_{0,\Ga_1}
+
\sum_{\al=1}^3 \norm{  X_\al }_{1} \right)
+
C\norm{\overline{\partial}( \Pi X) }_{0,\Ga_1}
\end{split}
\label{estimate_projection_N}
\end{align}
and

\begin{align}
\begin{split}
\norm{ \overline{\partial}{}^2 X_3 }_{0,\Ga_1}
\leq 
& 
\, \widetilde{\epsilon} \norm{X}_{2.5} 
+ 
\widetilde{\epsilon}
\norm{ \overline{\partial}{}^2 X }_{0,\Ga_1}
+
C\norm{\overline{\partial}{}^2 (\Pi X) }_{0,\Ga_1}
+ P(\norm{q}_{1.5,\Ga_1}, \norm{ \partial_t X(0)}_{0})  
+  \int_0^t P(\norm{\partial_t X}_0 ).
\end{split}
\label{estimate_projection_N_2_improved}
\end{align}
\end{lemma}


{\begin{proof}[Proof of Lemma~\ref{L05}]
Recalling \eqref{projection_identity}
\begin{align}
\begin{split}
(\Pi X)_3 = & \Pi_{3\la} X_\la 
=  (\de_{3\la} - g^{kl} \partial_k \eta_3 \partial_l \eta_\la ) X_\la 
= X_3 - g^{kl} \partial_k \eta_3 \partial_l \eta_\la X_\la,
\end{split}
\label{Pi_X_difference}
\end{align}
so that 
\begin{align}
\begin{split}
\norm{ \overline{\partial} ( (\Pi X)_3 - X_3 ) }_{0,\Ga_1}
= &
\norm{ 
\overline{\partial} ( g^{kl} \partial_k \eta_3 \partial_l \eta_\la X_\la )
}_{0,\Ga_1}
\\
\leq &
C \norm{ 
  g^{kl} \partial_k \eta_3 \partial_l \eta_\la  \overline{\partial}X_\la 
}_{0,\Ga_1}
+
C \norm{ 
 \overline{\partial}(  g^{kl} \partial_k \eta_3 \partial_l \eta_\la) X_\la 
}_{0,\Ga_1}.
\end{split}
\nonumber
\end{align}
Since
\begin{gather}
\left. g^{kl} \partial_k \eta_3 \partial_l \eta_\la \right|_{t=0} = 0
\label{eta_3_zero_time_zero}
\end{gather}
due to $\eta_3(0) = x_3$, we have 
\begin{align}
\begin{split}
\norm{ g^{kl} \partial_k \eta_3 \partial_l \eta_\la }_{1.5,\Ga_1}
= &\nnorm{ \int_{0}^{t}\partial_t (g^{kl} \partial_k \eta_3 \partial_l \eta_\la ) }_{1.5,\Ga_1}
\leq  C \int_0^t \norm{ \partial_t \overline{\partial} \eta}_{1.5,\Ga_1}
\leq  C \int_0^t \norm{v}_3
\leq  CM t.
\end{split}
\nonumber
\end{align}
Hence,
\begin{align}
\begin{split}
\norm{ 
  g^{kl} \partial_k \eta_3 \partial_l \eta_\la  \overline{\partial}X_\la 
}_{0,\Ga_1} \leq &
C \norm{ 
  g^{kl} \partial_k \eta_3 \partial_l \eta_\la 
}_{1.5,\Ga_1} \norm{  \overline{\partial}X_\la  }_{0,\Ga_1} 
\leq  C M t \sum_{\la=1}^3 \norm{  \overline{\partial}X_\la  }_{0,\Ga_1} .
\end{split}
\nonumber
\end{align}
Next,
\begin{align}
\begin{split}
&\norm{ 
 \overline{\partial}(  g^{kl} \partial_k \eta_3 \partial_l \eta_\la) X_\la 
}_{0,\Ga_1}
\leq 
C \norm{ 
 \overline{\partial}(  g^{kl} \partial_k \eta_3 \partial_l \eta_\la) 
}_{L^4(\Ga_1)}
\sum_{\la=1}^3 \norm{ X_\la }_{L^4(\Ga_1)}
\\
&
\qquad{}
\leq 
C \norm{ 
 \overline{\partial}(  g^{kl} \partial_k \eta_3 \partial_l \eta_\la) 
}_{0.5, \Ga_1}
\sum_{\la=1}^3 \norm{ X_\la }_{0.5,\Ga_1}
\leq 
CM t\sum_{\la=1}^3 \norm{ X_\la }_{1}.
\end{split}
\nonumber
\end{align}
The estimate \eqref{estimate_projection_N} now follows upon choosing $t$ sufficiently small.

We now turn to \eqref{estimate_projection_N_2}. From \eqref{Pi_X_difference}, we have
\begin{align}
\begin{split}
\overline{\partial}{}^2 (g^{kl} \partial_k \eta_3 \partial_l \eta_\la X_\la )
= &
\,
\overline{\partial}{}^2 g^{kl} \partial_k \eta_3 \partial_l \eta_\la X_\la 
+
2 \overline{\partial} g^{kl} \overline{\partial} \partial_k \eta_3 \partial_l \eta_\la X_\la 
+
2 \overline{\partial} g^{kl} \partial_k \eta_3 \overline{\partial}\partial_l \eta_\la X_\la 
\\
& 
+
2 \overline{\partial} g^{kl} \partial_k \eta_3 \partial_l  \eta_\la \overline{\partial} X_\la 
+
g^{kl} \overline{\partial}{}^2 \partial_k \eta_3 \partial_l \eta_\la X_\la 
+
2 g^{kl} \overline{\partial} \partial_k \eta_3 \overline{\partial} \partial_l \eta_\la X_\la 
\\
& 
+
g^{kl} \overline{\partial} \partial_k \eta_3 \partial_l \eta_\la \overline{\partial} X_\la 
+
g^{kl} \partial_k \eta_3 \overline{\partial}{}^2 \partial_l \eta_\la X_\la 
+
2 g^{kl} \partial_k \eta_3 \overline{\partial}\partial_l \eta_\la \overline{\partial} X_\la 
\\
&
+
g^{kl} \overline{\partial} \partial_k \eta_3 \partial_l \eta_\la \overline{\partial} X_\la 
+
g^{kl} \partial_k \eta_3 \partial_l \eta_\la \overline{\partial}{}^2 X_\la 
\\
 = &A_1 + \cdots +  A_{11}.
\end{split}
\nonumber
\end{align}
We estimate the terms similarly to those leading to \eqref{estimate_projection_N}.
We  make successive use of the following facts.
As before, the terms in $\overline{\partial} \eta_3$ are small (in appropriate norms) 
in light of \eqref{eta_3_zero_time_zero} and what immediately follows.
 But here also the terms $\overline{\partial}{}^2 \eta_\al$ 
are small (in appropriate norms), for $\al=1,2,3$, again because $\eta(0)$ is the identity, so that, 
as before
\begin{gather}
\overline{\partial}{}^2 \eta = \int_0^t \overline{\partial}{}^2 v.
\nonumber
\end{gather}
To proceed, we note
that from \eqref{metric_def} and \eqref{identity_g_inverse} we have
\begin{align}
\overline{\partial} g^{-1} & = Q(\overline{\partial} \eta) \overline{\partial}{}^2 \eta
\nonumber
\end{align}
and
\begin{align}
\overline{\partial}{}^2 g^{-1} & = Q(\overline{\partial} \eta) (\overline{\partial}{}^2 \eta)^2
+ Q(\overline{\partial} \eta) \overline{\partial}{}^3 \eta,
\nonumber
\end{align}
where we adopted a symbolic notation akin to Remark~\ref{remark_rational}.
We have
\begin{align}
\begin{split}
\norm{ A_1 }_{0,\Ga_1} 
\leq & 
C \norm{ \overline{\partial}{}^2 g^{-1} }_{0,\Ga_1} \norm{\overline{\partial} \eta_3 }_{1.5,\Ga_1}
\norm{ \overline{\partial} \eta }_{1.5,\Ga_1} \norm{ X }_{1.5,\Ga_1}
\\
\leq &  P(\norm{ Q(\overline{\partial} \eta )}_{1.5,\Ga_1})
 (\norm{ (\overline{\partial}{}^2 \eta )^2}_{0,\Ga_1} + 
\norm{\overline{\partial}{}^3 \eta}_{0,\Ga_1} ) \widetilde{\epsilon} \norm{X}_2
\\
\leq &
\, \widetilde{\epsilon} P
 (\norm{ \overline{\partial}{}^2 \eta }^2_{L^4(\Ga_1)} + 
\norm{ \eta}_{3,\Ga_1} )  \norm{X}_2
\\
\leq & \, \widetilde{\epsilon} P(\norm{q}_{1.5,\Ga_1} ) \norm{X}_2
\end{split}
\nonumber
\end{align}
and
\begin{align}
\begin{split}
\norm{A_2}_{0,\Ga_1} \leq & C \norm{ \overline{\partial}{}^2\eta }_{1.5,\Ga_1}
\norm{ \overline{\partial}{}^2 \eta_3}_{0,\Ga_1} \norm{X}_{1.5,\Ga_1}
\leq  \, \widetilde{\epsilon} P(\norm{q}_{1,\Ga_1}) \norm{X}_2.
\end{split}
\nonumber
\end{align}
The terms $A_3$ and $A_4$ are estimated similarly as
\begin{align}
\begin{split}
\norm{A_3}_{0,\Ga_1}  + \norm{A_4}_{0,\Ga_1}  
\leq & \, \widetilde{\epsilon} P(\norm{q}_{1,\Ga_1}) \norm{X}_2.
\end{split}
\nonumber
\end{align}
The term $A_5$ is different and it gives
the term without $\widetilde{\epsilon}$ in \eqref{estimate_projection_N_2},
\begin{align}
\begin{split}
\norm{ A_5 }_{0,\Ga_1} \leq & C \norm{ g^{-1} \overline{\partial}{}^3 \eta 
\overline{\partial} \eta X }_{0,\Ga_1} 
\leq  C \norm{\eta}_{3.5,\Ga_1} \norm{ X}_{1.5,\Ga_1}
\leq  P(\norm{q}_{1.5,\Ga_1} )\norm{X}_2.
\end{split}
\nonumber
\end{align}

Next,
\begin{align}
\begin{split}
\norm{ A_6}_{0,\Gamma_1} \leq & C \norm{ \overline{\partial}{}^2 \eta_3 }_{L^4(\Ga_1)}
\norm{ \overline{\partial}{}^2 \eta}_{L^4(\Ga_1)} \norm{X}_{1.5,\Ga_1} 
\leq   \, \widetilde{\epsilon} \norm{X}_2
\end{split}
\nonumber
\end{align}
and 
\begin{align}
\begin{split}
\norm{ A_7}_{0,\Gamma_1} \leq & C \norm{ \overline{\partial}{}^2 \eta_3 }_{L^4(\Ga_1)}
\norm{ \overline{\partial}{} \eta}_{L^4(\Ga_1)} \norm{\overline{\partial}X}_{L^4(\Ga_1)} 
\leq   \, \widetilde{\epsilon} \norm{ X}_2.
\end{split}
\nonumber
\end{align}
The terms $A_8$, $A_9$ and $A_{10}$ are handled similarly. Finally,
\begin{align}
\begin{split}
\norm{ A_{11} }_{0,\Ga_1} \leq & C \norm{ \overline{\partial} \eta_3 }_{1.5,\Ga_1}
 \norm{ \overline{\partial} \eta }_{1.5,\Ga_1} \norm{ \overline{\partial}{}^2 X }_{0,\Ga_1}
 \leq  \widetilde{\epsilon} \norm{ \overline{\partial}{}^2 X }_{0,\Ga_1}.
\end{split}
\nonumber
\end{align}
Combining the above gives 
\begin{align}
\begin{split}
\norm{ \overline{\partial}{}^2 X_3 }_{0,\Ga_1}
\leq 
& 
\, \widetilde{\epsilon} 
\sum_{\al=1}^3 \norm{ \overline{\partial}{}^2 X_\al }_{0,\Ga_1}
+
P(\norm{q}_{1.5,\Ga_1} ) \norm{  X }_{2} 
+
C\norm{\overline{\partial}{}^2( \Pi X )}_{0,\Ga_1}.
\end{split}
\label{estimate_projection_N_2}
\end{align} 
Using interpolation, Young's, and Jensen's inequalities
as in 
\eqref{estimate_partial_t_v_lower_order},
we readily
obtain 
\eqref{estimate_projection_N_2_improved}.
\end{proof}

\subsection{Eliminating the lower order terms\label{section_eliminating_lower_order}}

Our estimates so far contain some lower order terms on the right-hand sides that need to be eliminated.

In \eqref{estimate_v_2.5_boundary}, we have $\norm{\partial_t q}_1$. 
Arguing as in 
\eqref{estimate_partial_t_v_lower_order}
we obtain
\begin{gather}
\norm{\partial_t q}^2_1 \leq \widetilde{\epsilon} \norm{\partial_t q}_2^2 
+\ccPz
+ \int_0^t 
\ccP.
\label{partial_t_q_lower_order}
\end{gather}
Next, we look at the term  $P(\norm{v}_{2.5+\de}, \norm{q}_{1.5,\Ga_1})$. 
Up to here, $\de$ was any number between $0$ and $0.5$. 
Now we determine it.
By the Cauchy-Schwarz inequality, we have
\begin{gather}
P(\norm{v}_{2.5+\de}, \norm{q}_{1.5,\Ga_1}) \leq P(\norm{v}_{2.5+\de}) + P(\norm{q}_{1.5,\Ga_1})
,
\nonumber
\end{gather}
where
$P(\norm{v}_{2.5+\de}) $ is a linear combination of terms of the form $\norm{v}_{2.5+\de}^a$, where
$a \geq 1$ is integer. Since
\begin{gather}
\norm{v}_{2.5+\de} \leq \frac{1}{2} + \frac12 \norm{v}_{2.5+\de}^2,
\nonumber
\end{gather}
we may assume $a \geq 2$. The interpolation inequality gives
\begin{gather}
\norm{v}_{2.5+\de}^a \leq  C\norm{v}_3^{\fractext{(0.5-\de)a}{3}} 
\norm{v}_0^{\fractext{(2.5+\de)a}{3}}.
\label{interpolation_lower_order_1}
\end{gather}
Let $a_m$ be the largest of the powers $a$, and choose $\de \equiv \de_m$ sufficiently close to $0.5$ so that
\begin{gather}
\frac{ 6 }{a_m (0.5-\de_m) } > 1.
\nonumber
\end{gather}
Set
$
p_a  = \fractext{ 6 }{a (0.5-\de_m) }
$,
and note that $p_a \geq p_{a_m} > 1$, so that the conjugate exponent
$q_a$
is well-defined
and also greater than $1$. Therefore, we may
apply Young's inequality with epsilon to the right hand-side of \eqref{interpolation_lower_order_1}
to find
\begin{align}
\begin{split}
\norm{v}_{2.5+\de}^a \leq 
&
\, \widetilde{\epsilon} \norm{v}_3^2 + C  \norm{v}_0^b,
\end{split}
\nonumber
\end{align}
where
$
b = \fractext{(2.5+\de)a q}{3}
$.
Since $ a \geq 2$ and $q > 1$, we have $b>1$ and hence we may use Jensen's inequality to obtain
\begin{gather}
\norm{v}_0^b \leq P(\norm{v_0}, \norm{v}_{3.5,\Ga_1} ) 
+ \int_0^t \ccP.
\nonumber
\end{gather}
We therefore conclude that
\begin{align}
\begin{split}
P(\norm{v}_{2.5+\de}) \leq 
&
\, \widetilde{\epsilon} \norm{v}_3^2 + C +  \int_0^t 
\ccP.
\end{split}
\nonumber
\end{align}
For $P(\norm{q}_{1.5,\Ga_1})$, we simply note that 
$
\norm{q}_{1.5,\Ga_1} \leq C\norm{q}_{2.5+\widetilde{\de}}
$
for any $0 < \widetilde{\de} < 0.5$. Thus we can proceed exactly as in the estimate for $P(\norm{v}_{2.5+\de})$.
We finally conclude
\begin{align}
\begin{split}
P(\norm{v}_{2.5+\de}, \norm{q}_{1.5,\Ga_1} ) \leq 
&
\, \widetilde{\epsilon}( \norm{v}_3^2 + \norm{q}_3^2 ) + C 
+  \int_0^t \ccP.
\end{split}
\label{polynomial_lower_order}
\end{align}

\section{Proof of the main theorem}

We are now ready to combine all our estimates. 
This consists essentially of keeping track of the energy and
pressure estimates and verifying that they always give estimates
for higher order terms in terms time integrals 
of the terms to be bounded.
We only sketch the details since the Gronwall argument is
fairly standard.

{\begin{proof}[Proof of Theorem~\ref{main_theorem}]
From \eqref{estimate_partial_t_v_div_curl} and \eqref{estimate_partial_2_t_v_div_curl},
we see that we need $\norm{\partial_t v_3}_{2,\Ga_1}$
and 
$\norm{\partial^2_t v_3}_{1,\Ga_1}$. By interpolation, we have
\begin{align}
\begin{split}
\norm{ \partial_t v_3 }_{2,\Ga_1} \leq & C
\norm{ \overline{\partial}{}^2 \partial_t v_3 }_{0,\Ga_1} 
+ C \norm{ \partial_t v_3}_{0,\Ga_1}
\\ 
\leq &
C \norm{ \overline{\partial}{}^2 \partial_t v_3 }_{0,\Ga_1} 
+ P(\norm{v_0}_3,\norm{v_0}_{3.5,\Ga_1}) 
+  \int_0^t \ccP
\end{split}
\nonumber
\end{align}
and
\begin{align}
\begin{split}
\norm{ \partial^2_t v_3 }_{1,\Ga_1} \leq & C
\norm{ \overline{\partial} \partial^2_t v_3 }_{0,\Ga_1} 
+ C \norm{ \partial^2_t v_3}_{0,\Ga_1}
\\ 
\leq &
C \norm{ \overline{\partial} \partial^2_t v_3 }_{0,\Ga_1} 
+ \widetilde{\epsilon} \norm{\partial^2_t v}_{1.5} 
+  P(\norm{v_0}_3,\norm{v_0}_{3.5,\Ga_1}) 
+  \int_0^t \ccP.
\end{split}
\nonumber
\end{align}

Thus, it suffices to estimate
$\norm{ \overline{\partial}{} \partial^2_t v_3 }_{0,\Ga_1} $ 
and  $\norm{ \overline{\partial}{}^2 \partial_t v_3 }_{0,\Ga_1} $ which,
in light of \eqref{estimate_projection_N} and \eqref{estimate_projection_N_2_improved},
are estimated, up to some harmless lower order terms and terms with $\widetilde{\epsilon}$,
 in terms of $\norm{ \overline{\partial} (\Pi \partial^2_t v)}_{0,\Ga_1}$
and $\norm{ \overline{\partial}{}^2 (\Pi \partial_t v)}_{0,\Ga_1}$.
An estimate for these last two terms is given in  \eqref{partial_3_t_v_estimate}
and \eqref{partial_2_t_v_estimate}.

Therefore, combining the estimates of Proposition~\ref{proposition_pressure_estimates}
with \eqref{estimate_v_div_curl},
\eqref{estimate_partial_t_v_div_curl},  \eqref{estimate_partial_2_t_v_div_curl},
invoking  \eqref{partial_3_t_v_estimate}, \eqref{partial_2_t_v_estimate}, 
and \eqref{estimate_v_2.5_boundary}, eliminating the lower order terms with
\eqref{partial_t_q_lower_order} and \eqref{polynomial_lower_order}, 
and proceeding as in \cite{IgorMihaelaSurfaceTension}, we obtain
a Gronwall-type inequality as in 
\cite{DragomirGrownwall} (or cf.~\cite{CoutandShkollerFreeBoundary,KukavicaTuffahaVicol-3dFreeEuler,IgorMihaelaSurfaceTension}), which establishes Theorem~\ref{main_theorem}.
\end{proof}

\bibliographystyle{plain}
\bibliography{References.bib}

\end{document}